\documentclass[11pt]{amsart}

\usepackage{amsfonts,amsmath,latexsym,amssymb,verbatim,amsbsy,amsthm}
\usepackage{graphicx}
\usepackage{caption}
\usepackage{float}
\usepackage{wrapfig}
\usepackage{mathrsfs}
\usepackage{enumitem}

\usepackage[top=1in, bottom=1in, left=1in, right=1in]{geometry}

\usepackage[dvipsnames]{xcolor}

\usepackage[colorlinks=true, pdfstartview=FitV, linkcolor=RoyalBlue,citecolor=ForestGreen, urlcolor=blue]{hyperref}

\theoremstyle{plain}
\newtheorem{THEOREM}{Theorem}[section]

\newtheorem{theorem}[THEOREM]{Theorem}
\newtheorem{corollary}[THEOREM]{Corollary}
\newtheorem{lemma}[THEOREM]{Lemma}

\theoremstyle{definition}

\newtheorem{definition}[THEOREM]{Definition}

\theoremstyle{remark}

\def \a {\alpha}

\def \g {\gamma}
\def \d {\delta}

\def \e {\varepsilon}
\def \f {\varphi}
\def \l {\lambda}
\def \n {\nabla}

\def \D {\Delta}

\def \L {\Lambda}

\def \bu {{\bf u}}

\def \cA {\mathcal{A}}
\def \cB {\mathcal{B}}
\def \cC {\mathcal{C}}

\def \cF {\mathcal{F}}

\def \cM {\mathcal{M}}

\def \cP {\mathcal{P}}

\def \dH {\dot{H}}

\newcommand{\N}{\ensuremath{\mathbb{N}}}   
\newcommand{\Z}{\ensuremath{\mathbb{Z}}}   
\newcommand{\R}{\ensuremath{\mathbb{R}}}   
\newcommand{\T}{\ensuremath{\mathbb{T}}}   

\def \loc {\mathrm{loc}}

\def \sign {\mathrm{sgn}}

\def \lan {\langle}
\def \ran {\rangle}

\def \p {\partial}

\def \dx  {\, \mbox{d}x}

\def \dy  {\, \mbox{d}y}
\def \dz  {\, \mbox{d}z}

\def \ds  {\, \mbox{d}s}

\def \ddt  {\frac{\mbox{d\,\,}}{\mbox{d}t}}

\begin{document}


\title[Unidirectional Flocks: Critical case]{Global existence and limiting behavior of unidirectional flocks  for the fractional Euler Alignment system}
\author{Daniel Lear} 

\address{Department of Mathematics, Statistics, and Computer Science, University of Illinois at Chicago}

\email{lear@uic.edu}

\date{\today}



\thanks{\textbf{Acknowledgment.} The author is deeply indebted to Roman Shvydkoy for many discussions and suggestions.  The author is also grateful to Trevor Leslie for his careful reading of an early version of the manuscript.}

\begin{abstract}
In this note we continue our study of unidirectional solutions to  hydrodynamic Euler alignment systems
with strongly singular communication kernels $\phi(x):=|x|^{-(n+\a)}$ for $\a\in(0,2)$.
Here, we consider the critical case $\a=1$ and establish a couple of global existence results of smooth solutions, together with a full description of their long time dynamics. The first one is obtained via Schauder-type estimates under a null initial entropy condition and the other is a small data result.
In fact, using Duhamel's approach we get that any solution is almost Lipschitz-continuous in space.
We extend the notion of weak solution for $\a\in[1,2)$ and prove the existence of global Leray-Hopf solutions. Furthermore, we give an anisotropic Onsager-type criteria for the validity of the natural energy law for weak solutions of the system. Finally, we provide a series of quantitative estimates that show how far the density of the  limiting flock is from a uniform distribution depending solely on the size of the initial entropy.
\end{abstract}

\maketitle

\section{Introduction and statement of main results}
In this note we consider the following  hydrodynamic Euler alignment system  (EAS) for density $\rho(x,t)$ and velocity $\bu(x,t)=(u^{1}(x,t),\ldots,u^{n}(x,t))$  :
\begin{equation}\label{e:CSHydro}
(x,t)\in\mathbb{R}^{n}\times\mathbb{R}^{+}\qquad \left\{
\begin{split}
\partial_t \rho +\nabla\cdot (\rho \bu )&= 0, \\
\partial_t  \bu +\bu \cdot\nabla \bu &= \mathcal{L}_{\phi} (\rho  \bu ) -  \mathcal{L}_{\phi}(\rho)\bu,
\end{split}\right.
\end{equation}
subject to initial condition
$$\left( \rho(\cdot,t), \bu (\cdot,t)\right)|_{t=0}=(\rho_{0},\bu_{0}).$$

This system represents a multi-dimensional hydrodynamic version of the celebrated Cucker-Smale   flocking model introduced in \cite{CS2007a,CS2007b}, with a huge number of applications ranging from biology or robotics
to social sciences, see \cite{CCP2017,Vicsek-Zafeiris, S-book} for recent surveys and references therein. The system \eqref{e:CSHydro} is designed to describe the interaction between agents governed by laws of self-organization with communication protocol encoded into the kernel $\phi$.

In many realistic situations and applications, the communication among agents takes place in local neighborhoods induced by short-range kernels. Particularly interesting is the case of singular communication kernels. Clearly, singularity at the origin strongly emphasizes local communication.

For models with singular kernels given by $\phi(x):=|x|^{-(n+\a)}$  for  $0<\alpha<2$ the
operator $\mathcal{L}_{\phi}\equiv\mathcal{L}_{\a}$ becomes the (negative of) classical fractional Laplacian:
\[
\mathcal{L}_{\a}(f)(x)=-\Lambda^{\a}(f)(x)= \text{p.v.}\int_{\R^n}\frac{f(y)-f(x)}{|x-y|^{n+\a}}\, \dy    \qquad \Lambda^{\a}:=(-\Delta)^{\a/2},   \qquad 0<\a<2.
\]
The alignment term on the right hand side of the velocity equation in \eqref{e:CSHydro} is then given by the following singular integral operator:
\begin{equation}\label{e:alignmentalpha} 
\mathcal{C}_{\a}(\bu,\rho) := -\Lambda^{\a} (\rho  \bu ) +  \Lambda^{\a}(\rho)\bu=\text{p.v.}\int_{\R^n}\frac{\bu(y)-\bu(x)}{|x-y|^{n+\a}} \rho(y)\, \dy.   
\end{equation}
In view of no-vacuum condition $(\rho_0>0)$ necessary to develop a well-posedness theory (c.f. \cite{Arnaiz-Castro},\cite{Tan}), we consider the periodic domain $\T^n$, where a uniform lower bound on the density is compatible with finite mass.
When dealing with the $n$-dimensional torus, the term \eqref{e:alignmentalpha} can be expressed in terms of the periodized kernel
\[
\phi_{\a}(z):=\sum_{k\in \Z^n}\frac{1}{|z+2\pi k|^{n+\a}}, \qquad 0<\a<2
\]
In the rest of the manuscript, we assume that $\bu(\cdot,t)|_{\T^n}$ and $\rho(\cdot,t)|_{\T^n}$ are extended periodically onto the whole space $\R^n$. Taking into account the above, the alignment term \eqref{e:alignmentalpha} becomes a fractional elliptic operator:
\[
\mathcal{C}_{\a}(\bu,\rho)=\text{p.v.}\int_{\T^n}\left(\bu(x+z)-\bu(x)\right) \rho(x+z) \phi_{\a}(z) \dz,   
\]
with the density controlling uniform ellipticity. Written in this form,  system \eqref{e:CSHydro}  bear resemblance to the prototypical fractional Burgers equation with non-local and non-homogeneous dissipation. 

\subsection{A ``naive'' overview of EAS in 1D}

In the  sequel of papers \cite{ST1,ST2,ST3} Tadmor and Shvydkoy proved global existence of smooth solutions for the one-dimensional  system \eqref{e:CSHydro} with alignment term given by  \eqref{e:alignmentalpha} in the  full range  $0< \a < 2$.  At the same time, Do et. al. in \cite{DKRT2018} treated the range $0<\a<1$, where they proved global existence of smooth solutions by applying the method of modulus of continuity as in Kiselev et. al. \cite{KNV2007}. In both approaches the problem requires  refined and moderns tools from regularity theory of fractional parabolic equations, and reduces the problem to verification of  a continuation criterion either in terms of  $u_x \in L^1\left([0,T_0);L^{\infty}\right)$, \cite{ST1}, or in terms of $\rho_x \in L^1\left([0,T_0);L^{\infty}\right)$, \cite{DKRT2018}.

In addition, about the long time behavior of the solution, Tadmor and Shvydkoy proved in \cite{ST2} that all global and smooth solution converges exponentially fast to a flocking state, by which we understand alignment to a constant velocity $u \to \bar{u}$, and stabilization of density to a traveling wave
\begin{equation}\label{e:traveling}
\rho(x,t)\to\rho_\infty(x-t\bar{u}).
\end{equation}
Here the average velocity $\bar{u}$, is determined by the conserved mass and momentum. That is, we have
\[
\bar{u}:=\frac{\cP}{\cM},
\]
where
\[
\cM:=\int \rho_0(x) \dx, \quad \cP:=\int (\rho_0u_0)(x) \dx.
\] 
Note that although  the limiting velocity $\bar{u}$ is prescribed from the initial condition, the shape of the limiting density profile  $\rho_{\infty}$ is a quantity that emerge from the dynamics of the system. In order to gain insight about the limiting shape of $\rho_{\infty}$, Leslie and Shvydkoy  apply in \cite{LS-entropy} the Csisz\'ar-Kullback inequality to determine how far $\rho_{\infty}$ is from the uniform distribution $m=\tfrac{1}{2\pi}\cM$ in the $L^1(\T)$ metric.\\

Global well-posedness theory for these singular models has been mainly developed only in 1D due to presence of an additional conserved quantity
\begin{equation}\label{e:entropy1D}
e:=u_x -\Lambda^{\a}(\rho), \qquad e_t +(eu)_x=0.
\end{equation}
Due to this relation, one can compare the regularity of $u$ and $\rho$ and, using the compactness of the 1D torus, obtain a uniform in time positive lower bound on the density, thanks to which the alignment term on the right-hand side of \eqref{e:alignmentalpha} does not disappear. That nice relation unfortunately fails in higher dimension. 
In multi-dimensional settings  the corresponding quantity is given by
\[
e:= \nabla \cdot \bu -\Lambda^{\a}(\rho),
\]
and satisfies
\[
e_t+\nabla \cdot(\bu e)=\left(\nabla \cdot \bu \right)^2 -\text{Tr}[\left(\nabla \bu\right)^2].
\]
Lack of control on $e$ in this setting is part of the reason why in multiple dimensions the model has no developed a full regularity theory and only partial answers are known.

\subsection{Our new results for EAS in multi-D}
Global existence of smooth solutions for Euler alignment system in dimension 2 and higher  is an open problem. The three exceptions are:
\begin{itemize}
	\item \cite{Shv2018} For the full range $0<\a<2.$ A small initial data result for nearly aligned 
flocks with small initial velocity variations relative to its higher order norms.
	\item \cite{DMPW2019} For the range $1<\a<2$. A small initial data result where the deviation of
the initial density from a constant is sufficiently small with smallness formulated in Besov spaces.
	\item \cite{LS-uni1} For the range $1<\a<2$. A non-small initial data result taking advantage of the nice structure of unidirectional flocks. See below \eqref{e:ansatz} for a precise definition of this type of flows.
\end{itemize}
\textbf{Remark:} Here, we emphasize that the previous result \cite{LS-uni1} is only available in the subcritical case and it does not apply neither for the critical case  nor for the supercritical case.  One of the main reasons behind these possible different behaviors depending  on subcritical, critical and supercritical regimes is the required assumptions in order to apply H\"{older} regularization for drift-diffusion equation via Silvestre's result \cite{Silvestre_Holder}. In the case $\a\geq 1$, the velocity vector field is required to belong to the scale invariant class $L_t^{\infty}L_x^{\infty}.$ However, in the case $\a<1$, the velocity vector field is required to belong to the scale invariant class $L_t^{\infty}C_x^{1-\a}.$ \\

In this note, we will focus on the setting of unidirectional flows for the particular case $\a=1.$
Recall that unidirectional flows are given by 
\begin{equation}\label{e:ansatz}
\bu(x,t)=u(x,t)\mathbf{d}, \quad \mathbf{d}\in \mathbb{S}^{n-1}, \quad u:\R^{n}\times \R^{+}\rightarrow \R.
\end{equation}
The key point of this type of flows is that the same conservation law \eqref{e:entropy1D} holds for the entropy
\[
e:=\mathbf{d}\cdot \nabla u -\Lambda(\rho), \qquad \p_t e + \mathbf{d}\cdot\nabla (u e)=0,
\]
although in this case the entropy does not control the full gradient of the velocity. For simplicity, in view of rotational invariance of the EAS, we can postulate that \textbf{d} points in the direction of the $x_1$-axis. So, we can assume without loss of generality that
\[
\bu(x,t):=\lan u(x,t),0,\ldots,0 \ran \qquad \text{for} \quad u:\R^{n}\times \R^{+}\rightarrow \R.
\]
Note  that  the  non-trivial  component $u(x,t)$  may  depend  on  all  coordinates.   So,  unidirectional solutions exhibit features of a 1D flow, yet being on $\R^n$ represent solutions of a multi-dimensional system of scalar conservation laws:
\begin{equation}\label{e:CSHansatz}
(x,t)\in\mathbb{R}^n\times\mathbb{R}^{+}\qquad \left\{
\begin{array}{@{} l l @{}}
\partial_t \rho  +\partial_1(\rho u)&\hspace{-0.3 cm}=0, \\
\partial_t u+ \frac12\partial_1 (u^2)&\hspace{-0.3 cm}=-\Lambda (\rho  u ) +  \Lambda(\rho)u.
\end{array}\right. 
\end{equation}

As in \cite{ST1,ST2,ST3}, our first result is based on the fact that density
and momentum equations of \eqref{e:CSHansatz} fall under a general class of forced parabolic drift-diffusion equations  with (a priori) rough coefficients.
Our methodology will be to reduce the forceless problem, which correspond with null initial entropy $e_0=0$, to a known Schauder estimate together with the propagation of higher order regularity following the spirit of \cite{Silvestre_higher}. However, this argument does not provide a good quantitative estimate to conclude flocking. We can bypass this difficulty using an adaptation of the non-linear maximum principle as in Constantin and Vicol's proof for the critical SQG \cite{CV2012}, that allows us to completely describe the long time behavior.

The first result summarized in the following theorem covers the global regularity and flocking behavior for the critical case $(\a=1)$ under null initial entropy $e_0=0.$ 
\begin{theorem}[Global existence of classical solution  with null initial entropy]\label{t:nullentropy}
Suppose $m\geq 3$ and let $(u_0,\rho_0) \in H^{m+1}(\T^n) \times H^{m+1}(\T^n)$ with $\rho_0(x) > 0$ and $e_0(x)=0$ for all $x\in \T^n$. Then there exists a unique non-vacuous global in time solution to \eqref{e:CSHansatz} in the class
\begin{equation}\label{e:class}
u \in C_w([0,\infty); H^{m+1}(\T^n)) \cap  L^2([0,\infty); \dH^{m+ 3/2}(\T^n) ), \quad \rho \in C_w([0,\infty); H^{m+1}(\T^n)).
\end{equation}
Moreover, the solution obeys uniform bounds on the density
\begin{equation}\label{e:rhopositive}
\rho^{-} \leq \rho(x,t)\leq \rho^{+}, \quad t\geq 0,
\end{equation}
and strongly flocks, meaning $\exists \rho_{\infty}\in H^{m+1}$ such that
\begin{equation*}
\|u(t)-\bar{u}\|_{W^{2,\infty}(\T^n)} +\|\rho(\cdot, t) - \rho_{\infty}(\cdot -\bar{u}t) \|_{C^{\gamma}(\T^n)} \leq C e^{-\d t}, \qquad t>0,\quad  (0<\gamma<1). 
\end{equation*}
\end{theorem}

For general initial entropy, based on Duhamel's principle for a fractional parabolic equation, we can prove that density and velocity are almost Lipschitz in space (i.e. $C^{\g}(\T^n)$ for all $0<\g<1$).
This makes it rather natural to ask under what conditions does one has a global existence result.
In this direction, we establish a small initial data result where smallness is expressed only in terms of the initial amplitude of the solution.
To be precise, let us introduce some notation and terminology. We define the amplitude by
\begin{equation}\label{e:amplitude}
\cA(t):=\max_{x,y\in \T^n}|u(x,t)-u(y,t)|.
\end{equation}

\begin{theorem}[Global existence of classical solution with small initial amplitude]\label{t:smallamplitude} 
Suppose $m \geq \nolinebreak 3$ and let $(u_0,\rho_0) \in H^{m+1}(\T^n) \times H^{m+1}(\T^n)$ with $\rho_0(x) > 0$ for all $x\in \T^n$. There exists a universal positive constant $\e^{\star}\ll 1$ (small enough) such that if $\cA_0\leq \e^{\star},$ then there exists a unique non-vacuous global in time solution to \eqref{e:CSHansatz} in the class \eqref{e:class} which aligns and flocks exponentially fast.
\end{theorem}
The above result is an improvement of \cite{Shv2018} for the critical case, where the author required small initial amplitude relative to its higher order norms. 
Now, in the absence of a full answer to the question of global existence of classical solutions, one may wonder whether the
critical Euler alignment system \eqref{e:CSHansatz} at least possesses global weak solutions for generic initial data in $L^{\infty}(\T^n)-$spaces.
To do that, we follow the steps and  ideas introduced in \cite{Leslie-weak} for 1D, where Leslie construct global weak solutions as limits of regular ones. The main difference with our multidimensional setting is that we do not have global classical solutions in the critical case, but we will be able to extend the result using the avaible global regularity for the subcritical regime.  

\begin{theorem}[Global existence of weak solutions for $\a\in[1,2)$]\label{t:weak} Let $(\rho_0,u_0,e_0)\in (L^{\infty}(\T^n))^3$  with $\rho_0>0$ a.e. on $\T^n$ satisfying the compatibility condition \eqref{e:compatibility}.  Then there exists a global-in-time Leray-Hopf weak solution $(\rho,u,e)$ associated to the initial data $(\rho_0,u_0,e_0)$.
\end{theorem}
The limiting process we use to construct weak solutions exploits what it is know as a compactness argument. A fundamental step is the proof of certain ``compactness'' (which we satisfy by proving bounds in H\"{o}lder spaces) that allows to get strong convergence in suitable norms. As a consequence of the available uniform bounds of $u, \rho, \rho^{-1}$ and $e$, the H\"{o}lder regularization will survive the limiting process (with H\"{o}lder exponent $\g-\e$ for any $\e\in(0,\g)$ and $\g\in(0,1)$). Therefore, the constructed global weak solution becomes instantaneously almost Lipschitz in space for strictly positive time. However, whether there is unique weak solution remains unknown  and this is left for future research.\\

Furthermore,  we give an anisotropic Onsager-type criteria for the validity of the natural energy law for weak solutions of the system. We emphasize that the criteria we consider apply to any weak solutions, not just those weak
solutions which can be constructed as limit of regular ones.

\begin{theorem}[Energy equality for weak solutions]\label{t:EEweak}
For dimension $n=1,2,3$ or $4$ and $\a\in[1,2)$ any weak solution satisfies the natural energy laws \eqref{e:ee1} and \eqref{e:ee2}. For dimension $5\leq n\leq 9$ we have an unconditional result only if $\a\in[(n+2)/6,2).$ Otherwise, the energy equalities hold under the additional anisotropic condition:
\[
\rho\in L^3(0,T;\cB_{3,\infty}^{1/3}(\T^n)), \qquad u\in L^3(0,T;\cB_{3,c_0}^{1/3}(\T^n)).
\]
\end{theorem} 
A more precise description of the anisotropic spaces mentioned above can be found below in \eqref{e:aniBesov}. But for clarity and brevity we can understand $\cB_{p,r}^s(\T^n)$ as the set of periodic functions in $L^3(\T^n)$ and classical Besov regularity $B_{p,r}^s(\T)$ in the direction of the flow.\\

Finally, for any global (weak or classical) solution, we can obtain quantitative estimates of the long time behavior of the density in terms of the initial entropy $|e_0|_{\infty}.$
Following \cite{LS-entropy}, the thrust of our  result is to show that the latter is the parameter that controls
deviation from the uniform flock for any $L^p(\T^n)-$metric.

\begin{theorem}[Limiting flock in $L^p(\T^n)$ with $1\leq p <\infty$ for $\a\in[1,2)$]\label{t:limitingflock} 
Let $(\rho, u)$ be a global weak or classical solution to the system \eqref{e:CSHansatz} given by one of the above results. Then we have a bound for the difference between the limiting flock $\rho_{\infty}$ and the uniform distribution $m=(2\pi)^{-n}\cM$ given by
\[
|\rho_{\infty}-m|_{p} \lesssim |e_0|_{\infty}. 
\]
In particular, if $e_0=0$, then we have an exponential relaxation towards the uniform state:
\[
|\rho(t)-m|_{p} \lesssim |\rho_0-m|_{p}\, e^{-c t}.
\]
\end{theorem}
Although the conclusion of the above theorem and the main result of \cite{LS-entropy} share some similarities. It will be important to mention that their hypotheses are quite different. In first place, the result of \cite{LS-entropy} works for local and global kernels (smooth or singular) but only in the $L^1(\T)-$metric. In contrast, our previous result works only for global kernels (smooth or singular) as we required a lower bound on the density but in our favor it works for any $L^p(\T^n)-$metric with $1\leq p<\infty$.\\

\noindent
\textit{Notation:}
For convenience, to avoid clutter in computations, function arguments (time and space) will be omitted whenever they are obvious from context. Moreover, we use the notation $f\lesssim g$ when there exists a constant $C >0$ independent of the parameters of interest such that $f\leq C g$. We also use $|\cdot|_{L^p}$ or $|\cdot|_{p}$ to denote the classical $L^p$-norms, and $\| \cdot \|_X$ to denote all other norms.\\

\noindent
\textit{Organization:}
In Section \ref{s:nullentropy} we prove a global existence result for the critical Euler alignment system via fractional Schauder estimates under a null initial entropy condition, and provide higher order control estimates on solutions to prove a strong flocking result. In Section \ref{s:GWPsmallA}, as a direct application of the computations obtained in \cite{LS-uni1}, we get a global existence result under a small initial amplitude. 
We develop the theory of weak solutions for unidirectional flocks with $\a\in[1,2)$ in Section \ref{s:weaktheory} and we study the limiting behavior of any global (weak or classical) solution in Section \ref{s:Limitflock}.

\section{Global well-posedness and strong flocking with null entropy}\label{s:nullentropy}
According to  \cite[Theorem 1.1]{LS-uni1} we already have a local solution $(u,\rho)$ on time
interval $[0, T_0)$. In order to obtain a global classical solution for the critical case $\a=1$, we proceed in several steps. We start  collecting the uniform bounds on the density for any $0<\a<2$. Here, with an explicit dependence on the initial conditions, which will be important in the rest of the paper. 

Next, we invoke results from the theory of fractional parabolic equations to conclude that our solution gains H\"{o}lder regularity after a short period of time, and the H\"{o}lder exponent as well as the bound on the H\"{o}lder norm depend on the $L^{\infty}$ bound of the solution.

This result was a key point to prove  global existence of classical solutions and strong floking for the subcritical case $1<\a<2$ (see  \cite[Theorem 1.2]{LS-uni1}). However, this idea is not enough for the critical case $\a=1$, where we need to use fractional Schauder estimates to obtain higher order regularity.
To sum up, Schauder estimates are just a good quick way of obtaining global well-posendess and regularization, but not a good way to obtain long time estimates. To do that, we will need to come back to our pedestrian estimates proved in  \cite{LS-uni1} in order to obtain strong flocking to a uniform state.\\

\noindent {\sc Step 1: bounds on the density ($0<\a<2$).}  
We start by establishing uniform bounds on the density which depend only on the initial conditions.
First, recall that the ratio $q:=e/\rho$ satisfies the transport equation $q_t + u q_1=0.$
Starting from sufficiently smooth initial condition with $\rho_0$ away from vacuum we can assume that
\begin{equation}\label{e:q(t)=q(0)}
|q(t)|_{\infty}=|q_0|_{\infty}<\infty.
\end{equation}
So, we can write the continuity equation  as
\[
\rho_t + u \rho_1 =-q \rho^2-\rho \L^{\a}(\rho).
\]
Let us evaluate at a point $x^{+}(t)$ where the maximum of $\rho(\cdot,t)$, denoted by $\rho^{+}(t):=\rho(x^{+}(t),t)$, is reached. We obtain
\begin{align*}
\ddt \rho^{+}(t)&=-q(x^{+}(t),t)\left(\rho^{+}(t)\right)^2-\rho^{+}(t)\int_{\R^n} \left(\rho^{+}(t)- \rho(x^{+}(t)+z,t)\right)\frac{\dz}{|z|^{n+\a}}\\
&\leq |q_0|_{\infty}\left(\rho^{+}(t)\right)^2-\frac{\rho^{+}(t)}{r^{n+\a}}\int_{|z|<r} \left(\rho^{+}(t)- \rho(x^{+}(t)+z,t)\right)\dz\\
&\leq |q_0|_{\infty}\left(\rho^{+}(t)\right)^2-\frac{\rho^{+}(t)}{r^{n+\a}}\left(V_n(r)\rho^{+}(t)-\cM \right),
\end{align*}
where $V_n(r)$ denotes the $n$-dimensional volume of a ball of radius $r$. As $V_n(r)=C(n)r^n$, we get
\[
\ddt \rho^{+}(t)\leq \left[|q_0|_{\infty}-\frac{C(n)}{r^{\a}}\right]\left(\rho^{+}(t)\right)^2+\frac{\cM}{r^{n+\a}}\rho^{+}(t).
\]
Let us pick $r$ small enough so that $\frac{C(n)}{r^{\a}}\geq |q_0|_{\infty}+1$. Then 
\begin{equation}\label{e:upper}
\ddt \rho^{+}(t)\leq - \left(\rho^{+}(t)\right)^2 + \cM\left(\frac{|q_0|_{\infty}+1}{C(n)}\right)^{\frac{n+\a}{\a}} \rho^{+}(t),
\end{equation}
which establishes the upper bound by integration.

As to the lower bound we argue similarly. Let $\rho^{-}(t)$ be the minimum value of $\rho(\cdot,t)$ and $x^{-}(t)$ a point
where such value is achieved. We have 
\begin{align*}
\ddt \rho^{-}(t)&=-q(x^{-}(t),t)\left(\rho^{-}(t)\right)^2-\rho^{-}(t)\int_{\T^n} \left(\rho^{-}(t)- \rho(x^{-}(t)+z,t)\right)\phi_{\a}(z)\dz\\
&\geq -|q_0|_{\infty}\left(\rho^{-}(t)\right)^2-\phi_{\a}^{-}\,\rho^{-}(t)\int_{\T^n} \left(\rho^{-}(t)- \rho(x^{-}(t)+z,t)\right)\dz\\
&\geq -|q_0|_{\infty}\left(\rho^{-}(t)\right)^2-\phi_{\a}^{-}\,\rho^{-}(t)\left((2\pi)^n\rho^{-}(t)-\cM\right).
\end{align*}
Note that at this point the global communication of the model is crucial: $\phi_{\a}^{-}:=\inf_{z\in\T^n}\phi_{\a}(z)>0.$
Then
\begin{equation}\label{e:lower}
\ddt \rho^{-}(t)\geq -\left(|q_0|_{\infty}+\phi_{\a}^{-}(2\pi)^n\right) \left(\rho^{-}(t)\right)^2+\phi_{\a}^{-}\cM \rho^{-}(t),
\end{equation}
which establishes the lower bound by integration.\\
Up to this point, we have proved uniform lower and upper bounds for the density. That is,  there exists a couple of strictily positive constants $\rho^{\pm}>0$ such that  $\rho^{-}\leq \rho (t)<\rho^{+}$ for all $t\geq 0$. Now, we are going to see how this constants depends explicitly of the initial data $(\rho_0,u_0).$ To do that, we use the logistic equation $\dot{X}(t)= A X(t)[B-X(t)]$,
where  $A$ and $B$ are positive constants and $X(t)$ is a positive function, then
\begin{align*}
X(t)= \frac{B X(0)}{X(0)+(B-X(0))e^{-ABt}}.
\end{align*}
Consequently, writting \eqref{e:upper} and \eqref{e:lower} as a logistic differential inequality we obtain
\[
\rho^{+}(t)\leq \frac{c_1 \rho_0^{+}}{\rho_0^{+}+(c_1-\rho_0^{+})e^{-c_1 t}},  \qquad c_1:= \cM\left(\frac{|q_0|_{\infty}+1}{C(n)}\right)^{\frac{n+\a}{\a}},
\]
and
\[
\rho^{-}(t)\geq  \frac{c_2 \rho_0^{-}}{\rho_0^{-}+(c_2-\rho_0^{-})e^{-c_3 t}},  \qquad c_2:= \frac{\phi_{\a}^{-}\cM}{|q_0|_{\infty}+\phi_{\a}^{-}(2\pi)^n}, \quad c_3:=\phi_{\a}^{-}\cM.
\]
By the above, it is simple to check that
\[
\rho^{+}\leq \max\{\rho_0^{+},c_1\}\leq \max \left\lbrace \rho_0^{+},\cM\left(\frac{|e_0|_{\infty}+\rho_0^{-}}{C(n) \rho_0^{-} }\right)^{\frac{n+\a}{\a}} \right\rbrace,
\]
and
\[
\rho^{-}\geq \min\{\rho_0^-,c_2\}\geq \min \left\lbrace \rho_0^{-}, \frac{ \rho_0^{-}\phi_{\a}^{-}\cM}{|e_0|_{\infty}+\rho_0^{-}\phi_{\a}^{-}(2\pi)^n} \right\rbrace.
\]
\noindent {\sc Step 2: bounds on the entropy ($0<\a<2$).}  
As an immediate consequence of the uniform bound on the density and \eqref{e:q(t)=q(0)} we have a uniform global bound on the entropy $|e(t)|_{\infty}<\infty$. More specifically, we have
\begin{equation}\label{e:uniformentropy}
|e(t)|_{\infty}\leq \frac{\rho^+}{\rho^-}|e_0|_{\infty}.
\end{equation}

\noindent {\sc Step 3:  H\"{o}lder regularization ($1\leq \a<2$).}
The parabolic nature of the density and  momentum equation is an essential structural feature of the system that has been used in all of the preceding works. Using the condition $u_1=e+\L^{\a}(\rho)$ we can write
\begin{equation}\label{e:density}
\rho_t + u \rho_1 +e \rho =-\rho \L^{\a}(\rho).
\end{equation}
Similarly, one can write the equation for the momentum $m=\rho u$:
\begin{equation}\label{e:momemtum}
m_t+u m_1 +e m=-\rho\Lambda^{\a}(m).
\end{equation}
Note that in both cases the drift $u$ is bounded a priori due to the maximum principle for each scalar velocity component.  Hence, density and momemtum equations \eqref{e:density}, \eqref{e:momemtum} falls under the general class of  fractional parabolic equations with bounded drift and force:
\begin{equation}\label{e:integro-diff}
w_t+u\cdot\nabla w+f=\mathcal{L}_{\a}(w), \qquad \mathcal{L}_{\a}(w)(x,t):=\int_{\R^n}\mbox{K}(x,z,t)\left(w(x+z,t)-w(x,t)\right) \dz,
\end{equation}
with a diffusion operator associated with the singular kernel $\mbox{K}(x,z,t):=\rho(x,t)|z|^{-(n+\a)}$ which is even with respect to $z$. The bounds on the density provide uniform ellipticity bounds on
the kernel:
\[
(2-\a)\frac{\lambda}{|z|^{n+\a}}\leq \mbox{K}(x,z,t)\leq (2-\a)\frac{\Lambda}{|z|^{n+\a}}.
\]
The most common assumption in the literature is that for all $x$ and $t$, the kernel $\mbox{K}$ is comparable pointwise in terms of $z$ to the kernel of the fractional Laplacian.

Regularity of these equations has been the subject of active research in recent years. In particular, the result of Silvestre \cite{Silvestre_Holder}, see also Schwab-Silvestre \cite{Schwab-Silvestre},  which provides H\"{o}lder regularization bound for some $\gamma>0$ given by
\begin{align}\label{e:gamma-rho}
\|\rho\|_{C^{\gamma}(\T^n \times [T/2,T))} &\lesssim |\rho|_{L^{\infty}(\T^n \times [0,T))}+|\rho e|_{L^{\infty}(\T^n \times [0,T))},\\
\|m\|_{C^{\gamma}(\T^n \times [T/2,T))} &\lesssim |m|_{L^{\infty}(\T^n \times [0,T))}+|m e|_{L^{\infty}(\T^n \times [0,T))}, \nonumber
\end{align}
and
\begin{equation}\label{e:gamma-u}
\|u\|_{C^{\gamma}(\T^n \times [T/2,T))} \leq C\left(|\rho|_{L^{\infty}(\T^n \times [0,T))},|u|_{L^{\infty}(\T^n \times [0,T))}\right),
\end{equation}
where the latter inequality follows from \eqref{e:gamma-rho} since $\rho$ is bounded below.
Since the right hand side  of \eqref{e:gamma-rho} and \eqref{e:gamma-u} is uniformly bounded on time we have obtained uniform bound on $C^{\gamma}$-norm starting, by rescaling, from any positive time. 

\textbf{Remark:} Note that the solutions $u$ and $\rho$ are H\"{o}lder continuous on compact sets of $\T^n \times (0, \infty)$. 
As a consequence of the uniform boundedness of $u, \rho, \rho^{-1}$ and $e$, the H\"{o}lder exponent $\g$ can be taken to be independent of $T$. Up to this point, we only know that solution is $C^{\g}_{t,x}$ with $\g>0.$\\

\noindent {\sc Step 4:  Schauder estimates.} For simplicity, in the rest of the section, we will focus only on the critical case $\a=1$ where the result of \cite{LS-uni1} does not apply. The key point is to note that if we initially assume that $e_0 = 0$, it will remain zero as long as the solution exists, and things get much simpler.
Notice that this translate in the forceless case of \eqref{e:integro-diff} with H\"older coefficients, due to the apriori estimates \eqref{e:gamma-rho} and \eqref{e:gamma-u}. Now, following the spirit of Silvestre \cite{Silvestre_higher}, we are in position to apply the recent result of Dong-Jin-Zhang \cite[Theorem 1.4.]{Dong-Jin-Zhang} to obtain that the solution becomes immediately
differentiable with H\"older continuous derivatives. That is, we have
\begin{align}\label{e:1+gamma-rho}
\|\rho\|_{C^{1+\gamma}(\T^n \times [3T/4,T))} &\lesssim \|\rho\|_{C^{\gamma}(\T^n \times [T/2,T))} \lesssim |\rho|_{L^{\infty}(\T^n \times [0,T))},\\
\|m\|_{C^{1+\gamma}(\T^n \times [3T/4,T))} &\lesssim\|m\|_{C^{\gamma}(\T^n \times [T/2,T))} \lesssim |m|_{L^{\infty}(\T^n \times [0,T))}, \nonumber
\end{align}
and
\begin{align}\label{e:1+gamma-u}
\|u\|_{C^{1+\gamma}(\T^n \times [3T/4,T))} &\leq C\left(\|\rho\|_{C^{\gamma}(\T^n \times [T/2,T))},\|u\|_{C^{\gamma}(\T^n \times [T/2,T))} \right)\\
& \leq C\left(|\rho|_{L^{\infty}(\T^n \times [0,T))},|u|_{L^{\infty}(\T^n \times [0,T))}\right), \nonumber
\end{align}
where, as before, the latter inequality follows from \eqref{e:1+gamma-rho} since $\rho$ is bounded below. Since the right hand side  of \eqref{e:1+gamma-rho} and \eqref{e:1+gamma-u} is uniformly bounded on time we have obtained uniform bound on $C_{t,x}^{1+\gamma}$-norm starting, by rescaling, from any positive time.

This immediately translates into the fact that the solution fulfills the continuation criterion and the proof of global 
existence is complete for $\a=1$ under condition $e_0=0.$\\

After this, a natural question is to understand what happens in the case of general initial entropy. Here, there is a noticeable difference  between 1D and the multi-D setting. In 1D, since $\rho$ is $C^{\g}(\T)$ and bounded away from zero this implies that $e\in C^{\g}(\T)$ (see \cite[Remark 6.1]{ST1}). With this in mind, we have density and momemtum equation with drift and force in $C^\g(\T)$, that will be enough to apply the result of Dong-Jin-Zhang \cite[Theorem 1.4.]{Dong-Jin-Zhang} and obtain a different and shorther proof of global existence for $\a=1$ in 1D.
However, in multi-D we can not apply that result because we get that regularity only in the direction of the flow, that is $e\in C^{\g}_{x_1}(\T^n).$\\

In the general case ($e_0(x)\neq 0$ at some point $x\in \T^n$) we can prove that the solution is almost Lipschitz in space (i.e. $C^{\g}(\T^n)$ for all $0<\g<1$). To do that, we interpret the forced case as a perturbation of the forceless case through  Duhamel's formula. Let $w$ be the solution of
\begin{equation}\label{e:forceless}
w_t+u\cdot\nabla w=\mathcal{L}_{\a}(w), \qquad  (x,t)\in \T^n \times [0,T],
\end{equation}
with H\"older coefficient given by \eqref{e:gamma-rho} and \eqref{e:gamma-u}. Note that if $w$ is a solution to \eqref{e:forceless} on time interval $[0,T]$ then so is $w_{\l}(x,t):=w(\l x, \l t)$  on time interval $[0,T/\l]$ with $u_{\l}(x,t):=u(\l x, \l t)$ and $\mbox{K}_{\l}(x,z,t):=\mbox{K}(\l x,z,\l t).$ Stating the H\"older bound for $w_{\l}$ we readily obtain, for any $\g<1$, that
\begin{equation}\label{e:noforce}
\|w\|_{C^{\g}(\T^n)}(t)\leq \frac{C}{t^{\g}}|w|_{L^{\infty}(\T^n\times[0,T))}, \qquad t\in [T/2,T].
\end{equation}
Let us rewrite \eqref{e:noforce} in operator form. The equation \eqref{e:forceless} generates an strongly continuous
evolution family $\{\textbf{W}(t,s)\}_{t\geq s}$, where $\textbf{W}(t, s)w(s)$  give us the solution at time $t \geq s$ with initial condition $w(s)$ at time $t=s$. According to \eqref{e:noforce} we have
\begin{equation}\label{e:Wbound}
\|\textbf{W}(t,s)w(s)\|_{C^{\g}(\T^n)}\leq \frac{C}{(t-s)^{\g}}|w|_{L^{\infty}(\T^n \times[s,T))}, \qquad T/2\leq s \leq t \leq T.
\end{equation}
Let us now go back to the original forced equation \eqref{e:integro-diff} and treat the force as a source term. The solution to \eqref{e:integro-diff} by Duhamel's Principle can be written as
\begin{equation*}\label{e:duhamel}
w(t)=\textbf{W}(t,0)w_0 + \int_0^t \textbf{W}(t,s)f(s)\ds.
\end{equation*}
In view of \eqref{e:Wbound},
\begin{align}\label{e:inter1}
\|w(t)\|_{C^{\g}(\T^n)}&\leq \frac{C}{t^{\g}}|w_0|_{\infty}+C\int_0^t \frac{|f|_{L^{\infty}(\T^n\times[s,T))}}{(t-s)^{\g}}\ds\\
&\leq  \frac{C}{t^{\g}}|w_0|_{\infty}+\frac{C}{1-\g} \frac{t}{t^{\g}}|f|_{L^{\infty}(\T^n \times [0,T])}, \nonumber
\end{align}
and consequently we get
\begin{equation}\label{e:final}
\|w\|_{C^{\g}(\T^n)}(t)\leq \frac{C}{t^{\g}}\left( |w_0|_{\infty}+ |f|_{L^{\infty}(\T^n \times [0,T])} \right), \qquad t\in [T/2,T].
\end{equation}
Note that the constant $C$ (which can change at each line) deteriores as $\g \to 1^{-}.$
Therefore, coming back to our original problem, this provides almost  Lipschitz regularity in space for all $0<\gamma<1$ and all $t\in[T/2,T)$ given by
\begin{align}\label{e:Lipschitz-rho}
\|\rho(t)\|_{C^{\gamma}(\T^n )} &\lesssim |\rho|_{L^{\infty}(\T^n \times [0,T))}+|\rho e|_{L^{\infty}(\T^n \times [0,T))},\\
\|m(t)\|_{C^{\gamma}(\T^n )} &\lesssim |m|_{L^{\infty}(\T^n \times [0,T))}+|m e|_{L^{\infty}(\T^n \times [0,T))}, \nonumber
\end{align}
and
\begin{equation}\label{e:Lipschitz-u}
\|u(t)\|_{C^{\gamma}(\T^n )} \leq C\left(|\rho|_{L^{\infty}(\T^n \times [0,T))},|u|_{L^{\infty}(\T^n \times [0,T))}\right),
\end{equation}
where the latter inequality follows, as before, from \eqref{e:Lipschitz-rho} since $\rho$ is bounded below.  Since the right hand side  of \eqref{e:Lipschitz-rho} and \eqref{e:Lipschitz-u} is uniformly bounded on time we have obtained uniform bound on $C_{x}^{\gamma}$-norm starting, by rescaling, from any positive time.\\


\noindent {\sc Step 5: Strong Flocking.} Recall that we have proved above for the case $e_0=0$ that the following holds for all $T > 0$:
\begin{align*}
\|\rho\|_{C^{1+\gamma}(\T^n \times [3T/4,T))} &\lesssim  |\rho|_{L^{\infty}(\T^n \times [0,T))}\\
\|u\|_{C^{1+\gamma}(\T^n \times [3T/4,T))} &\lesssim \left(|\rho|_{L^{\infty}(\T^n \times [0,T))},|u|_{L^{\infty}(\T^n \times [0,T))}\right).
\end{align*}
Since the right hand sides are uniformly bounded on the entire line we have obtained uniform bounds on $C^{1+\g}(\T^n)$-norm starting from time $t = 3T/4$. Since we are concerned with long time dynamics let us reset initial time to $t = 3T/4$, and allow ourselves to assume that $C^{1+\g}(\T^n)$-norms are bounded from time $t = 0.$ Then the following uniform bound holds:
\begin{equation}\label{e:uniform-grad}
\sup_{t\geq 0} |\n \rho|_{\infty}<\infty, \qquad \sup_{t\geq 0} |\n u|_{\infty}<\infty.
\end{equation}

\noindent
{\sc Step 5.1: Control over $|\n u|_{\infty}$.}
In order to complete the proof of Theorem \ref{t:nullentropy}, we proceed to establish strong flocking for the velocity.
We already know from the previous step that $|\n u|_{\infty}$ remains uniformly bounded. However, this argument does not provide a good quantitative estimate to conclude flocking. We will seek more precise estimates with the fact that now we already know that $|\n u|_{\infty}$ remains uniformly bounded. So, we go back to equation (35) of \cite{LS-uni1}:
\begin{align}\label{e:finalgradu}
\p_t |\nabla u|_{\infty}^2 &\lesssim C\, |\nabla u|_{\infty}^2\left( |\nabla \rho|_{\infty}+|\nabla u|_{\infty}\right) +\e |\nabla \rho|_{\infty}^{3} + c_{\e} - (1-\e)D(\nabla u)(x_{\star})-(1-\e)\frac{|\nabla u|_{\infty}^{3}}{\cA(t)}, \nonumber
\end{align}
where
\[
D(\nabla u)(x_{\star}):=\int_{\R^n}\frac{|\n u(x_{\star}+z)-\n u(x_{\star})|^2}{|z|^{n+1}} \mbox{d}z,
\]
with $x_{\star}\equiv x_{\star}(t),$ the maximal point of $\n u(\cdot,t).$ That is, a point such that $\n u(x_{\star}(t),t)=|\n u (t)|_{\infty}.$
Now, using the uniform bound on $|\n \rho|_{\infty}$ from \eqref{e:uniform-grad} we obtain the bound
\[
\p_t |\nabla u|_{\infty}^2 \lesssim C\, |\nabla u|_{\infty}^2 +C\, |\nabla u|_{\infty}^3 + C_{\e} - (1-\e)D(\nabla u)(x_{\star})-(1-\e)\frac{|\nabla u|_{\infty}^{3}}{\cA(t)}. \nonumber
\]
Now, to take advantage of dissipation, we apply an improvement on a nonlinear maximum principle bound of \cite{CV2012} in the case of small amplitudes, which appear originally on \cite[Lemma 3.4]{ST2}:
\begin{lemma}
There is an absolute constant $c>0$ such that for all $B>0$ one has
\[
D(\n u)(x)\geq B |\n u (x)|^2 -c B^3 \cA^2(t).
\]
\end{lemma}
In view of the above lemma, we have
\[
\p_t |\nabla u|_{\infty}^2 \lesssim -\left( B(1-\e)-C\right)|\n u|_{\infty}^2-(1-\e-C\cA(t))\frac{|\nabla u|_{\infty}^{3}}{\cA(t)}+C_{\e}+(1-\e)c B^3 \cA^2(t).
\]
Moreover, using the uniform bounf of $|\n u|_{\infty}$ we have 
\[
C_{\e}\leq \e \frac{|\n u|_{\infty}^3}{\cA(t)}+ \tilde{C}_{\e} \cA(t). 
\]
Putting all together and taking $B\gg 1$ and $t \gg 1$ such that $B(1-\e)-C>0$ and $1-2\e-C\cA(t)>0$ we have proved
\[
\p_t |\nabla u|_{\infty}^2 \lesssim -|\nabla u|_{\infty}^2 +\cA(t).
\]
In particular, it implies exponential rate of convergence to zero as times goes to infinity.
Finally, proving exponential decay of $|\n^2 u|_{\infty}$ follows similar estimates, and will be omitted here for the sake of brevity. We will refer to \cite{ST2} for full details in 1D.

\noindent
{\sc Step 5.2: Strong Flocking for the density.}
Now, to establish strong flocking for the density, we pass to the moving frame $x-\bar{u}t$ and write the continuity equation $\p_t\rho+\p_1(\rho u)=0$ in the new coordinates. Then, $\tilde{\rho}(x,t):=\rho(x_{1}+t\bar{u},x_{-},t)$ where $x\equiv(x_1,x_{-})$ satisfies
$$\partial_{t}\tilde{\rho}(x,t)+\tilde{\rho}(x,t)\,\partial_{1} u(x_{1}+t\bar{u},x_{-},t) +\partial_{1}\tilde{\rho}(x,t)(u(x_{1}+t\bar{u},x_{-},t)-\bar{u})=0.$$
According to the previously established strong flocking for the velocity we have that $|\partial_{t}\tilde{\rho}|_{\infty}=E(t)$, where $E(t)$ denotes an exponential decaying quantity. This shows that $\tilde{\rho}(\cdot,t)$ is Cauchy in time, which proves that there exists a unique limiting state $\rho_{\infty}(\cdot)$ such that $|\tilde{\rho}(\cdot,t)-\rho_{\infty}(\cdot)|_{\infty}=E(t).$
Hence, returning to our original coordinates, this can be written in terms of $\rho(x,t)$ and $\bar{\rho}(x,t):=\rho_{\infty}(x_{1}-t\bar{u},x_{-},t)$ as 
$|\rho(\cdot,t)-\bar{\rho}(\cdot)|_{\infty}= E(t).$
Since $\nabla \rho$ is uniformly bounded, this also shows that $\bar{\rho}$ is Lipschitz. Finally, convergence
in $C^{\gamma}$ for any $0<\gamma<1$ follows by interpolation.

\section{Global well-posedness and strong flocking with small initial amplitude}\label{s:GWPsmallA}
For any initial entropy $e_0,$ the only available global existence result of classical solutions in the multi-dimensional critical case $(\a=1)$ is due to Shvydkoy \cite{Shv2018}, under a small condition on the  initial amplitude relative to its higher order norms of the initial data. 

The main goal of this section is to obtain a new global existence result of classical solutions under a weaker condition on the initial amplitude. Something in the spirit of there exist a universal positive constant $\e^{\star}$, with $\e^{\star}\ll 1$ small enough, such that the solutions of \eqref{e:CSHansatz} is global in time if the initial amplitude is smaller that $\e^{\star}.$ We will refer to the statement of Theorem \ref{t:smallamplitude} for a more precise description of the result. Lastly, note that one disadvantage of our result, in comparison with \cite{Shv2018}, is that does not work for the supercritical case.\\

The starting point will be the equations (24) and (35) of \cite{LS-uni1}:
\begin{equation}\label{e:finalgradrho}
\p_t |\nabla \rho |_{\infty}^2 \lesssim C \left(|\nabla \rho|_{\infty} |\nabla e|_{\infty} + |\nabla u|_{\infty} |\nabla \rho|_{\infty}^2\right)- D_{\a}(\nabla \rho)(x^{\star})- |\nabla \rho|_{\infty}^{3},
\end{equation}
\begin{equation}\label{e:finalgradu}
\p_t |\nabla u|_{\infty}^2 \lesssim C\, |\nabla u|_{\infty}^2\left( |\nabla \rho|_{\infty}+|\nabla u|_{\infty}\right) +\e |\nabla \rho|_{\infty}^{3} + c_{\e} - (1-\e)D(\nabla u)(x_{\star})-(1-\e)\frac{|\nabla u|_{\infty}^{3}}{\cA(t)}. 
\end{equation}
In order to complete the proof, we combine \eqref{e:finalgradrho} and \eqref{e:finalgradu} to obtain  that
\begin{align}\label{e:grad(rho+u)}
\p_t \left[|\nabla \rho |_{\infty}^2+|\nabla u |_{\infty}^2\right] &\lesssim C \left(|\nabla \rho|_{\infty} |\nabla e|_{\infty} +|\nabla u|_{\infty}^{2} |\nabla \rho|_{\infty}+ |\nabla u|_{\infty} |\nabla \rho|_{\infty}^2 +|\nabla u|_{\infty}^3\right)+c_{\e}\\
& \quad  -(1-\e)\frac{|\nabla u|_{\infty}^{3}}{\cA(t)}-(1-\e)|\nabla \rho|_{\infty}^{3} \nonumber \\
& \quad - (1-\e)D(\nabla u)(x_{\star}) -D(\nabla \rho)(x^{\star}). \nonumber
\end{align}
The above expression \eqref{e:grad(rho+u)} emphasizes the fact that the point at which each $D$-term is evaluated is different. 
Without losing generality, we can discard the last two terms of the previous expression and focus our attention on the cubic terms. Only by interpolation we get:
\begin{align*}
|\nabla u|_{\infty}^2 |\nabla \rho|_{\infty}&\leq \e  \frac{|\nabla u|_{\infty}^{3}}{\cA(t)}   + c_{\e}|\nabla \rho|_{\infty}^{3}\cA^2(t),\\
|\nabla u|_{\infty} |\nabla \rho|_{\infty}^2 &\leq \e  \frac{|\nabla u|_{\infty}^{3}}{\cA(t)}   + c_{\e}|\nabla \rho|_{\infty}^{3}\cA^{1/2}(t).
\end{align*}
Consequently, all the cubic terms can be bounded  above by
\begin{equation}\label{e:cubic}
|\nabla u|_{\infty}^{2} |\nabla \rho|_{\infty}+ |\nabla u|_{\infty} |\nabla \rho|_{\infty}^2 +|\nabla u|_{\infty}^3 \leq  (2\e+ \cA(t))\frac{|\nabla u|_{\infty}^{3}}{\cA(t)}  +2 c_{\e} \cA^{1/2}(t)|\nabla \rho|_{\infty}^{3}.
\end{equation}
So, putting \eqref{e:cubic} into \eqref{e:grad(rho+u)} we arrive at
\begin{align*}
\p_t \left[|\nabla \rho |_{\infty}^2+|\nabla u |_{\infty}^2\right] &\leq C |\nabla \rho|_{\infty} |\nabla e|_{\infty}+ c_{\e}\\
& \quad  -(1-3 \e-\cA(t))\frac{|\nabla u|_{\infty}^{3}}{\cA(t)}-(1- \e- c_{\e}\cA^{1/2}(t))|\nabla \rho|_{\infty}^{3}. \nonumber
\end{align*}
Now, due to the maximum principle for the velocity, we have that $\cA(t)\leq \cA_0$ and taking $A_0\ll 1$ small enough such that the following bounds hold
\[
1-3 \e-\cA_0>0, \qquad \text{and} \qquad 1- \e- c_{\e}\cA_0^{1/2}>0,
\]
we have proved
\[
\p_t \left[|\nabla \rho |_{\infty}^2+|\nabla u |_{\infty}^2\right] \leq C |\nabla \rho|_{\infty} |\nabla e|_{\infty}+ c_{\e}.
\]
To close the argument, we recall  that $|\nabla e (t)|_{\infty} \lesssim |\nabla \rho (t)|_{\infty}+ |\nabla u|_{L^1_{t} L^{\infty}_{x}}$ which forces to have an apriori control in time for the gradient of the velocity. We bypass this obstacle taking into account the fact that the ratio $q:=e/\rho$ satisfies the transport equation $q_t +u  q_1=0$. Consequently, we get
\begin{align}\label{e:rho+u+q}
\p_t \left[|\nabla \rho |_{\infty}^2+|\nabla u |_{\infty}^2 +|\nabla q|_{\infty}^2\right] &\lesssim C \left[|\nabla \rho|_{\infty} |\nabla e|_{\infty}+|\nabla u|_{\infty} |\nabla q|_{\infty}\right]+ c_{\e}. 
\end{align}
In addition, by the definition of $q$ and the uniform bounds previously proved for density and entropy, we trivially have that
$|\nabla e|_{\infty}\lesssim |\nabla \rho|_{\infty}+|\nabla q|_{\infty}$. Therefore, the quadratic term of \eqref{e:rho+u+q} can be bounded directly as  
\[
|\nabla \rho|_{\infty} |\nabla e|_{\infty}+|\nabla u|_{\infty} |\nabla q|_{\infty}\lesssim |\nabla \rho|_{\infty}^2+|\nabla u|_{\infty}^2+|\nabla q|_{\infty}^2.
\]
Then,  we have obtained uniform bound  for $\rho, u, q \in L^{\infty}([0,T_0);\dot{W}^{1,\infty})$  by integration. This  fulfills the continuation criterion of \cite[Theorem 1.1]{LS-uni1} and the proof of global existence in time of classical 
solutions of \eqref{e:CSHansatz} for the critical case $(\a=1)$ is complete.

Next, we complement this result by a strong flocking statement. We already know from the previous step that $|\n \rho|_{\infty}$ and $|\n u|_{\infty}$ remains uniformly bounded for all time.  However, this argument does not provide a good quantitative estimate on $|\n u|_{\infty}$ to conclude flocking. We will
seek more precise estimates with the fact that now we already know that $| \n \rho|_{\infty}$ remains uniformly
bounded. So, we go back to {\sc Step 5} of Section \ref{s:nullentropy} to conlude that, under the uniform bounds \eqref{e:uniform-grad}, the same result follows.

\section{Global existence of weak solutions}\label{s:weaktheory}
Given the difficulty of obtaining a general result for classical solutions (without any extra initial condition) for the case of critical alignment. Now, we develop a weak formulation theory for $\a=1.$
Once a reasonable definition of weak solution is given, to prove global existence one
usually exploits what it is know as a compactness argument, which consists mainly in the following two steps: 
\begin{enumerate}
	\item proving the existence of a sequence of relatively smooth approximating solutions satisfying appropriate
uniform estimates;
	\item proving that limits of these approximating solutions are effectively
weak solution of the problem under consideration. 
\end{enumerate}

For the interested reader, we refer to \cite{Leslie-weak}, where Leslie construct weak solutions as
limits of global classical solutions in 1D. The main difference with our result in multi-D is that there is no a full developed theory of global existence for classical solutions, only the results mentioned above.
We bypass this difficulty and construct weak solutions as limits of classical solutions for the subcritical regime. In fact, we extend and develop the weak theroy of unidirectional flows for any $\a\in[1,2).$\\


In order to obtain a reasonable notion of weak solution, we start checking some of the conditions satisfied for the classical ones. We recall that classical solutions satisfy the energy equalities: 
\begin{equation}\label{e:ee1}
\int_{\T^n} \rho(t)^2\dx + \frac12\int_0^t \int_{\T^n\times \R^n} \left(\rho(x)+\rho(y)\right) \frac{|\rho(x)-\rho(y)|^2}{|x-y|^{n+\a}}\dy\dx\ds = \int_{\T^n} \rho_0^2\dx - \int_0^t \int_{\T^n} e\rho^2\dx \ds, 
\end{equation}
and
\begin{equation}\label{e:ee2}
\frac12 \int_{\T^n} \rho u^2(t)\dx + \int_0^t \int_{\T^n \times \R^n} \rho(x)\rho(y) \frac{|u(x)-u(y)|^2}{|x-y|^{n+\a}}\dy\dx\ds = \frac12 \int_{\T^n} \rho_0 u_0^2\dx.
\end{equation}

Under non-vacuum condition $0<\rho^-\leq\rho(t)\leq \rho^+$, using Sobolev--Slobodeckij spaces, the above equalities prove that $\rho, u$ are bounded in $L^2(0,T; H^{\a/2}(\T^n))$ by a constant depending only on $\cM$, $T$, and the $L^\infty$ norms of $u_0$, $\rho_0$, $\rho_0^{-1}$, and $e_0$. Finally, since $L^\infty(\T^n)\cap H^{\a/2}(\T^n)$ is an algebra, we have that $\rho u$ is also bounded in $L^2(0,T;H^{\a/2}(\T^n))$. \\

{\sc Definition of weak solution.} For weaker notions of a solution, we include $e$ as part of our definitions. 
To write down a weak formulation, it is helpful to use the following system:
\begin{equation}\label{e:weaksystem}
\left\{
\begin{array}{@{} l l @{}}
\partial_t \rho  +\partial_1(\rho u)&\hspace{-0.3 cm}=0, \\
\p_t u +u e &\hspace{-0.3 cm}=-\L^{\a}(\rho u).
\end{array}\right. 
\end{equation}
\begin{definition}
Let $\a\in[1,2)$ and $(\rho_0, u_0, e_0)\in (L^{\infty}(\T^n))^3$ with $\rho_0>0$ a.e. on $\T^n$ satisfying the compatibility condition:
\begin{equation}\label{e:compatibility}
\int_{\T^n} e_0 \f + u_0 \p_1 \f +\rho_0 \L^{\a}(\f)\dx =0 \qquad \text{for all} \quad  \f\in C^{\infty}(\T^n).
\end{equation}
We will say that $(\rho, u, e)$ is a weak solution on the interval $[0, T]$ with initial data $(\rho_0, u_0, e_0)$,
if
\begin{itemize}
\item $u$, $\rho$, $\rho^{-1}$, and $e$ all belong to $L^\infty(0,T;L^\infty(\T^n))$.
\item$u$ and $\rho$ belong to $L^2(0,T;H^{\a/2}(\T^n))$. 
\item $(u,\rho,e)$ satisfies the weak form of \eqref{e:weaksystem}, for all $\f\in C^\infty(\T^n\times [0,T])$ and a.e. $t\in [0,T]$:
\begin{equation}
\label{e:weakv}
\int_{\T^n} u(t)\f(t)\dx - \int_{\T^n} u_0 \f(0)\dx - \int_0^t \int_{\T^n} u\p_t \f \dx \ds
= -\int_0^t \int_{\T^n} ue\f +  \L^{\a/2}(\rho u) \L^{\a/2}(\f) \dx\ds, 
\end{equation}
\begin{equation}
\label{e:weakd}
\int_{\T^n} \rho(t)\f(t)\dx - \int_{\T^n} \rho_0\f(0)\dx - \int_0^t \int_{\T^n} \rho \p_t \f \dx \ds
 = \int_0^t \int_{\T^n} \rho u \p_1\f \dx\ds.
\end{equation}
\item The compatibility condition \eqref{e:compatibility} propagates in time, in the sense that 
\begin{equation}
\label{e:compat2}
\int_0^T \int_{\T^n} e \f + u \p_1\f + \rho \L^{\a}(\f)  \dx \ds 
 = 0, \quad \text{ for all } \f \in C^\infty(\T^n\times [0,T]).
\end{equation}
\end{itemize}
Note that weak formulations \eqref{e:weakv}, \eqref{e:weakd}, \eqref{e:compat2} work for  $ \f\in C^{1}\left(0,T;C_{x_1}^1(\T^n)\cap L^{\infty}(\T^n)\cap H^{\a/2}(\T^n)\right)$.
\end{definition}

\begin{definition}
A weak solution $(\rho, u, e)$ on $\T^n\times [0,T]$ is a Leray-Hopf weak solution if the following inequalities hold for almost every $s \in (0, T)$ and every $t \in (s, T)$:
\begin{equation}\label{e:ei1}
\int_{\T^n} \rho^2(t)\dx + \frac12\int_s^t \int_{\T^n\times \R^n} (\rho(x)+\rho(y)) \frac{|\rho(x)-\rho(y)|^2}{|x-y|^{n+\a}}\dy\dx\ds \leq \int_{\T^n} \rho^2(s)\dx - \int_s^t \int_{\T^n} e\rho^2\dx \ds, 
\end{equation}
and
\begin{equation}\label{e:ei2}
\frac12 \int_{\T^n} \rho u^2(t)\dx + \int_s^t \int_{\T^n \times \R^n} \rho(x)\rho(y) \frac{|u(x)-u(y)|^2}{|x-y|^{n+\a}}\dy\dx\ds \leq \frac12 \int_{\T^n} \rho u^2(s)\dx.
\end{equation}
\end{definition}


{\sc Properties of weak solution.}
Let $(u, \rho, e)$ be a weak solution on the time interval $[0,T]$ associated to the initial data $(u_0, \rho_0, e_0)\in (L^\infty(\T^n))^3 $.  The purpose of the following list is to collect some simple facts about general weak solution, namely
\begin{itemize}
	\item The quantity $e$ satisfies a weak form of $\p_t e +\p_1(ue)=0$.  That is, for all $\f \in C^\infty(\T^n\times [0,T])$ and a.e. $t\in [0,T]$, we have
	\begin{equation}
	\label{e:ewk}
	\int_{\T^n} e(t)\f(t)\dx - \int_{\T^n} e_0\f(0)\dx - \int_0^t \int_{\T^n} e \p_t \f \dx \ds = \int_0^t \int_{\T^n} ue \p_1\f \dx\ds.
	\end{equation}
	\item The solution $(u,\rho, e)$ converges weak-$*$ in $L^\infty(\T^n)$ to the initial data.
	\item The weak time derivative of $u$ is a well-defined element of $L^2(0,T;H^{-\a/2}(\T^n))$; the weak time derivatives of $\rho$ and $e$ are well-defined elements of $L^\infty(0,T; H^{-1}(\T^n))$.
	
	\item The weak solution $(u,\rho)$ belongs to $C(0,T,L^2(\T^n)).$ 
\end{itemize}
As its proof is analog to 1D, we skip it and refer the interested reader to \cite{Leslie-weak} for more details.\\

{\sc Construction of a weak solution.}
In this section, for any $\a\in[1,2),$ we construct a global weak solution as a subsequential limit of a global classical solution for a regularized (subcritical $\a \leadsto \a+\e>1$) system with mollified initial data, as the mollification parameter $\e$ tends to zero, by using the Aubin-Lions's compactness lemma. 

\begin{lemma}[Aubin-Lions's compactness lemma, \cite{Aubin-Lions}]
Let $X,Y,Z$ be three Banach spaces with $X\subset Y\subset Z.$ Suppose the embedding $X\subset Y$ is compact and the embedding $Y\subset Z$ is continuous.  Assume $p,r\in [1,\infty]$, and define for $T>0$ the following space: 
	\[
	E = \{v\in L^p(0,T;X): \p_t v\in L^r(0,T;Z)\}.
	\]
	\begin{enumerate}
		\item If $p<\infty$, then the embedding $E\subset L^p(0,T; Y)$ is compact.
		\item If $p=\infty$ and $r>1$, then the embedding $E\subset C([0,T],Y)$ is compact.
	\end{enumerate}
\end{lemma}	

In the notation of the Aubin-Lions's compactness lemma, we set 
\[
X_T = C^{\g}(\T^n), 
\quad 
Y = C(\T^n),
\quad 
Z = H^{-1}(\T^n),
\]
where $\g\in(0,1)$ is the H\"{o}lder exponent and
\[
E_{\d,T} = \{v\in L^\infty(\d,T;C^{\g}(\T^n)): \p_t v\in L^2(\d,T;H^{-1}(\T^n))\}.
\]
The conclusion is then that the embedding $E_{\d, T} \subset C([\d,T]; C(\T^n))$ is compact for any $T>\d>0$.

Choose $(\rho_0, u_0, e_0)\in  (L^\infty(\T^n))^3$ satisfying $\rho_0>0$ a.e. on $\T^n$ and the compatibility condition \eqref{e:compatibility} and $\a\in[1,2)$. For that, we regularize the initial data $(\rho_0,u_0)$ as
\begin{align*}
\rho_0^\e&:=\eta_\e \ast \rho_0,\\
u_0^\e&:=\eta_\e \ast u_0,
\end{align*}
where $\eta_\e(x):=\e^{-1}\eta(x \e^{-1})$ with $\eta\in C^{\infty}(\R^n)$ is a standard mollifier supported in $|x|<1.$
Let $(\rho^\e, u^\e)$ denote solution associated to the subcritical system
\begin{equation}\label{e:regularized}
\left\{
\begin{split}
\partial_t \rho^{\e} +\p_1(\rho^{\e} u^{\e} )&= 0, \\
\partial_t  u^{\e} +\tfrac{1}{2}\p_1(u^{\e})^2  &= \L^{\a+\e} (\rho^{\e}  u^{\e} ) -  \L^{\a+\e}(\rho^{\e})u^{\e},
\end{split}\right.
\end{equation}
with $1<\a +\e<2,$ subject to initial condition
$$\left( \rho^{\e}(\cdot,t), u^{\e} (\cdot,t)\right)|_{t=0}=(\rho^{\e}_{0},u^{\e}_{0}).$$
As $\rho_0>0$ a.e. on $\T^n$, the regularized initial data $\rho^{\e}_0$ satisfies the no-vacuum condition $\rho^{\e}_0(x)>0$ for all $x\in\T^n$ and consequently we have proved, using  \cite[Theorem 1.2]{LS-uni1}, that there exists a unique non-vacuous global in time classical solution of \eqref{e:regularized} for all $0<\e\ll 1.$
Notice that $e^\e := \p_1 u^\e - \L^{\a}(\rho^\e)$ satisfies that $e_0^\e  = e^\e(0)$ automatically.\\

\noindent
\emph{Claim:} The sequences $\rho^\e$ and $u^{\e}$ are bounded in $E_{\d,T}$ for any $T>\d>0$.
\begin{proof}
	Fix $T>\d>0$. In order to prove the claim, one needs to prove the following two statements: 
	\begin{enumerate}
		\item $\rho^\e$ and $u^\e$ are $\e-$uniformly bounded sequences of $L^\infty(\d,T;C^{\g}(\T^n))$.
		\item $\p_t \rho^\e$ and $\p_t u^\e$ are $\e-$uniformly bounded sequences of $L^2(\d,T;H^{-1}(\T^n))$.
	\end{enumerate}
\emph{Proof of (1).}
We have proved on Section \ref{s:nullentropy} via fractional parabolic regularity that $\rho^\e$ and $u^\e$ are $\g-$H\"{o}lder continuous on compact sets of $\T^n\times (0,\infty)$. Precisely, we have
\begin{align*}
\|\rho^{\e}\|_{C^{\g}(\T^n \times [\d,T))} &\lesssim C_{\d,T}\left(|\rho_0^{\e}|_{\infty},|e_0^{\e}|_{\infty} \right), \\
\|u^{\e}\|_{C^{\g}(\T^n \times [\d,T))} &\lesssim C_{\d,T}\left(|\rho_0^{\e}|_{\infty},|u_0^{\e}|_{\infty} \right).
\end{align*}
As usual, $\|\eta_{\e}\|_{L^1(\R^n)}=1$ and using Young's convolution inequality we finally get
\[
\|\rho^{\e}\|_{C^{\g}(\T^n \times [\d,T))} \lesssim C_{\d,T}\left(|\rho_0|_{\infty},|e_0|_{\infty} \right), 
\]
and
\[
\|u^{\e}\|_{C^{\g}(\T^n \times [\d,T))} \lesssim C_{\d,T}\left(|\rho_0|_{\infty},|u_0|_{\infty} \right).
\]
Therefore $\{\rho^{\e}, u^{\e}\}_{\e>0}$ are $\e-$uniformly  bounded sequences in $L^\infty(\d,T;C^{\g}(\T^n)).$\\

\noindent
\emph{Proof of (2).}
As we have seen this result for weak solutions, it also applies to classical solutions  and consequently we have that  $\p_t \rho^{\e} \in L^\infty(0,T; H^{-1}(\T^n))$ and $\p_t u^{\e}\in L^2(0,T;H^{-(\a+\e)/2}(\T^n)).$ More specifically, we have
\begin{align*}
|\p_t \rho^{\e} |_{L^{\infty}(0,T; H^{-1}(\T^n))} &\lesssim |\rho^{\e} u^{\e}|_{L^\infty(0,T;L^{\infty}(\T^n))},\\
|\p_t u^{\e} |_{L^2(0,T;H^{-(\a+\e)/2}(\T^n))} &\lesssim | u^{\e} e^{\e}|_{L^\infty(0,T;L^{\infty}(\T^n))}+|\rho^{\e} u^{\e}|_{L^2(0,T;H^{(\a+\e)/2}(\T^n))}.
\end{align*}
By the bounds proved in Section \ref{s:nullentropy} we get that the $L_t^{\infty}L_x^{\infty}$ norms can be bounded above by a quantity that
depends only on the $L^{\infty}(\T^n)$ norms of $\rho_0^{\e}, u_0^{\e}$ and $e_0^{\e}.$  But these $L^{\infty}(\T^n)$ norms are bounded by those of $\rho_0, u_0$ and $e_0.$

For the $L_t^2 H_x^{(\a+\e)/2}$ norm, using that solution $\rho^{\e}, u^{\e}$ satisfy the energy equalities \eqref{e:ee1},\eqref{e:ee2} with $\a\leadsto \a+\e$, we get that $\rho^{\e}, u^{\e}$ are bounded in $L^2(0,T; H^{(\a+\e)/2}(\T^n))$ by a constant depending only on $|\rho_0^{\e}|_{1}$, $T$, and the $L^\infty(\T^n)$ norms of $u_0^{\e}$, $\rho_0^{\e}$ and $e_0^{\e}$.
Finally, since $L^\infty(\T^n)\cap H^{(1+\e)/2}(\T^n)$ is an algebra, we have that $\rho^{\e} u^{\e}$ is also bounded in $L^2(0,T;H^{(1+\e)/2}(\T^n))$ by a constant independent \nolinebreak of \nolinebreak $\e$. 

Besides, as $H^{-(\a+\e)/2}(\T^n)\subset H^{-1}(\T^n)$ we get that both sequences $\{\p_t\rho^{\e},\p_t u^{\e}\}_{\e>0}$ are $\e-$uniform bounded in $L^2(0,T;H^{-1}(\T^n)).$
\end{proof}

Now, as the above holds for any $T>\d>0$, applying the Aubin-Lions-Simon compactness lemma, we can now choose a subsequence $\{\e_k\}_{k\geq 1}$, tending to zero as $k\to \infty$, such that $u^{\e_k}$ and $\rho^{\e_k}$ converge (strongly) in $C([2^{-N},2^N];C(\T^n))$, with $N$ any natural number.  Using a standard diagonal argument, we obtain a further subsequence, which we continue to denote by $\e_k$, such that $u^{\e_k}$ and $\rho^{\e_k}$ converge to functions $u$ and $\rho$, respectively, in $C_{\loc}((0,\infty);C(\T^n))$.

Notice that we have proved that  $e^\e$ is bounded in $L^\infty(\T^n\times [0,T])$ and both $u^\e$ and $\rho^\e$ are bounded in $L^2(0,T;H^{(\a+\e)/2}(\T^n))\subset L^2(0,T;H^{\a/2}(\T^n))$.  Therefore, as a consequence of the Banach-Alaoglu theorem (see \cite{Banach-Alaoglu}), each of these sequences has a subsequence (still denoted $\e_k$) that converges weakly to some limit. That is,  $e^{\e_k}$ converges weak-$*$ in $L^\infty(\T^n\times [0,T])$ to some $e\in L^\infty(\T^n\times [0,T])$, and $u^{\e_k}$ and $\rho^{\e_k}$ converge weakly in $L^2(0,T;H^{\a/2}(\T^n))$ to $u$ and $\rho$. Then we can use a diagonal argument as above to send $T\to \infty$. To summarize, there exists a subsequence $\{\e_k\}$ and a triple $(u, \rho, e)$, such that as $k\to \infty$, we have 
\[ 
u^{\e_k}\to u \text{ and } \rho^{\e_k} \to \rho \text{ strongly in } C_\loc((0,\infty); C(\T^n));
\]
\[
u^{\e_k} \rightharpoonup u \text{ and } \rho^{\e_k} \rightharpoonup \rho \text{ weakly in } L^2_{\loc}((0,\infty); H^{\a/2}(\T^n));
\]
\[
e^{\e_k} \stackrel{*}{\rightharpoonup} e \text{ weak-}*\text{ in } L^\infty_{\loc}(\T^n\times [0,\infty)).
\]
Now $(u^{\e_k}, \rho^{\e_k})$ is a classical solution (therefore $(u^{\e_k}, \rho^{\e_k}, e^{\e_k})$ is a weak solution) for each $k$. We can therefore consider each term in each equation of the weak formulation and easily see that the above convergences guarantee that $(u,\rho, e)$ satisfies the weak formulation. This completes the existence part of Theorem \ref{t:weak}.  The construction gives H\"older continuity on compact sets of $\T^n\times (0,\infty)$.  Indeed, if $\g$ is the H\"older exponent associated to the interval $[0,T]$ as in Section \ref{s:nullentropy}, then for any $\widetilde{\g}\in (0,\g)$, the convergences $u^{\e_k}\to u$ and $\rho^{\e_k}\to \rho$ can be taken in $L^\infty(\d,T;C^{\widetilde{\g}}(\T^n))$ for any $\d>0$.\\

{\sc Energy Inequality for Constructed Solutions.}
We have that the solutions constructed above satisfy \eqref{e:ei1} and \eqref{e:ei2}, then there exists a Leray-Hopf type global weak solution to the system. The energy inequalities  follows from the  energy equalities for regular enough solutions, then we pass to the limit $k\to \infty$ in the sequence $(u^{\e_k}, \rho^{\e_k}, e^{\e_k})$ from the proof of existence above.
For the sake of brevity, as the argument follows mutatis mutandis with straightforward modifications, we refer the reader to \cite{Leslie-weak} for further details in 1D about the proof of inequalities \eqref{e:ei1}, \eqref{e:ei2}.  \\

{\sc Energy equality for weak solutions.}
Now, we give conditions which guarantee that the energy equalities \eqref{e:ee1} and \eqref{e:ee2} hold for weak solutions. Here, we emphasize that the criteria apply to any weak solution, not just those weak solutions constructed as
limit of regular ones.

Before starting, we note that it turns out to be easier to work with the momentum equation when
proving \eqref{e:ee2}. But, due to the limited regularity, first we must check that such a
formulation is valid.\\ 

When $(\rho, u,e)$ is a weak solution, we can make sense of the expression $\rho  \cC_{\a}(\rho, u)$ in a weak sense.  Define $X_{\a}:= H^{\a/2}(\T^n)\cap L^\infty(\T^n)$, and denote by $\rho  \cC_{\a}(\rho, u)$  the element of the dual $X_{\a}^*$ given by
\begin{equation}\label{e:weakmomentum}
\langle \rho \cC_{\a}(\rho, u), \f \rangle_{X_{\a}^*, X_{\a}} =  \int_{\T^n}  -\L^{\a/2}(\rho u) \L^{\a/2}(\rho \f) + \L^{\a/2}(\rho) \L^{\a/2}(\rho u \f) \dx.
\end{equation}
Recall that if $(\rho, u,e)$ is a weak solution on the interval $[0,T]$, by definition we have that $\rho$ and $u$ belongs to  
$L^2(0,T;H^{\a/2}(\T^n))$ and $L^{\infty}(0,T;L^{\infty}(\T^n)).$ So, in particular both belongs to $L^2(0,T;X_{\a}).$
Now, by the algebra property of this space $X_{\a},$ the above expression \eqref{e:weakmomentum} is well-defined.

\begin{lemma}\label{l:moment}
Let $(u, \rho, e)$ be a weak solution  on the time interval $[0,T]$. Then for each test function $\f\in C^{\infty}(\T^n\times [0,T])$ and a.e. $t\in [0,T]$, we have that 
\begin{equation}
\label{e:moment}
\begin{split}
\int_{\T^n} \rho u \f(t) & \dx - \int_{\T^n} \rho_0 u_0 \f(0)\dx - \int_0^t \int_{\T^n} \rho u \p_t \f(s)\dx\ds \\ & = \int_0^t \int_{\T^n} \rho u^2 \p_1\f \dx\ds + \int_0^t \langle \rho \cC_{\a}(\rho, u), \f \rangle_{X_{\a}^*, X_{\a}}\ds.
\end{split}
\end{equation}
\end{lemma}
Before proving this result, we need to take into account another not entirely trivial step by checking that we can substitute a
mollified in space weak solution into the weak formulations of density \eqref{e:weakd} and velocity \eqref{e:weakv} as a test function. This fact is not straightforward since a priori weak solutions may not have enough time regularity.
By an approximation argument (see \cite{S-energy}) one can show that for
any weak solution, the relationships \eqref{e:weakd} and \eqref{e:weakv} hold for all test funciton $\f$ that are smooth and localized in space, but only weakly Lipschitz in time. As we have seen before that weak solutions belongs to $C(0,T;L^2(\T^n)),$ we can apply Littlewood-Paley projections of weak solutions as test functions.\\

Next we present the classical Littlewood-Paley theory which plays an important role in the proof of our result. For a more detailed description on this theory we refer readers to the books \cite{fourier}, \cite{fourier2}.
Define length scales $\lambda_q := 2q$. Let us fix a nonnegative
radial function $\chi \in C_0^{\infty}(B(0,1))$ such that
\[
\chi(\xi):=\begin{cases}
1 \qquad \text{for } |\xi|\leq 3/4, \\
0 \qquad \text{for } |\xi|\geq 1.
\end{cases}
\]
Let us define a dyadic partition of unity given by $\phi(\xi):=\chi(\xi/2)-\chi(\xi)$ and
\[
\phi_q(\xi):=\begin{cases}
\phi(\lambda_q^{-1}\xi) \qquad &\text{for } q\geq 0,\\
\chi(\xi)\qquad &\text{for } q=-1.
\end{cases}
\]
Now, for a given tempered (periodic) distribution  $f$ we consider the Littlewood-Paley projections in
the following way
\begin{equation}\label{e:LP}
\D_q f:=\cF^{-1}(\phi_q(\xi)\cF f),
\end{equation}
and
\[
S_{Q} f:=\sum_{q=-1}^{Q}\D_q f = \cF^{-1}(\chi(\lambda_{Q+1}^{-1}\xi)\cF f),
\]
where $\cF$ and $\cF^{-1}$ denote the Fourier transform and inverse Fourier transform for $\T^n$:
\begin{equation*}
\cF (f)(\xi):=\int_{\T^n}f(x)e^{-2\pi i x\cdot \xi}\dx,\qquad \cF^{-1} (g)(x):=\sum_{\xi\in \Z^n}g(\xi)e^{2\pi i x\cdot \xi}.
\end{equation*}
So, we have $f=\lim_{Q\to \infty}S_Q f$ in the sense of distributions. Let us recall the definition of Besov spaces.
The Besov space $B_{p,r}^s(\T^n)$ ($s\in \R, 1\leq p,r \leq \infty$) is the space of tempered
distributions  whose corresponding norm, defined by 
\begin{equation}\label{e:besov}
\|f\|_{B_{p,r}^s(\T^n)}:= \left\| \l_q^s |\D_q f|_{p} \right\|_{\ell_q^r},
\end{equation}
if finite. We also define $B_{p,c_0}^s(\T^n)$  to be the space of tempered distributions  such that
$$\limsup_{q\to \infty} \l_q^s |\D_q f|_{p} = 0,$$
together with the norm inherited from  $B_{p,\infty}^s(\T^n).$ See \cite{fourier} for the standard properties of these spaces. \\

Now we have all the ingredients to  prove the above lemma. We refer to \cite{Leslie-weak} for the proof in 1D. We include the sketch of the analogous proof in multi-D  for completeness.
\begin{proof}[Proof of Lemma \ref{l:moment}] Substitute the test function $S_Q (S_Q (\rho) \f)$ into the weak velocity equation \eqref{e:weakv} and substitute the test function $S_Q (S_Q (u) \f)$ into the weak density equation \eqref{e:weakd}. Finally, using the compatibility condition
\[
S_Q e= S_Q \p_1 u-\L^{\a}(S_Q \rho),
\]
to eliminate $S_Q \p_1 u$ from the weak density equation we arrive (after some manipulations) to \eqref{e:moment}.
\end{proof}

In \cite{Leslie-weak}, it was shown that if  $u\in L^3(0,T; B^{1/3}_{3,c_0}(\T))$ and $\rho\in L^3(0,T; B^{1/3}_{3,\infty}(\T))$  is a weak solution to the 1D  Euler alignment system, then the solution conserves energy. That is, \eqref{e:ee1} and \eqref{e:ee2} hold. We follow a similar program in this section to obtain an anisotropic Onsager class condition for unidirectional flows in multi-D.

In order to do that, we define an anisotropic Besov space via an anisotropic Littlewood-Paley projection adapted to the unidirectional flow. Without lost of generality, in view of rotational invariance of the system, we will continue assuming that the flow moves in the $x_1-$direction.
So, we define 
\begin{equation}\label{e:LPunid}
\D_q^{x_1} f:=\cF^{-1}(\phi_q(\xi_1)\cF f),
\end{equation}
which is the Littlewood-Paley projection with dyadic bloc in the horizontal Fourier variable
and
\[
S_{Q}^{x_1} f:=\sum_{q=-1}^{Q}\D_q^{x_1} f = \cF^{-1}(\chi(\lambda_{Q+1}^{-1}\xi_1)\cF f).
\]
We define an auxiliary anisotropic Besov space $\tilde{\cB}_{p,r}^s(\T^n)$ ($s\in \R, 1\leq p,r \leq \infty$). That is the space of tempered distribution whose norm, defined by
\[
\|f\|_{\tilde{\cB}_{p,r}^s(\T^n)}:= \left\| \l_q^s |\D_q^{x_1} f|_{p} \right\|_{\ell_q^r},
\]
is finite. We also define $\tilde{\cB}_{p,c_0}^s(\T^n)$  to be the space of tempered distributions  such that
$$\limsup_{q\to \infty} \l_q^s |\D_q^{x_1} f |_{p} = 0,$$
together with the norm inherited from  $\tilde{\cB}_{p,\infty}^s(\T^n).$ Now, we have all the ingredients to define the anisotropic space $\cB_{p,r}^s(\T^n)$ adapted to the problem with $s\in \R, 1\leq p,r \leq \infty$:
\begin{equation}\label{e:aniBesov}
\cB_{p,r}^s(\T^n):=\tilde{\cB}_{p,r}^s(\T^n)\cap L^3(\T^n).
\end{equation}
This is just the set of periodic functions in $L^3(\T^n)$ and Besov regularity $B_{p,r}^s(\T)$ in the $x_1$-variable.\\


Next, we derive an energy budget relation associated to the $x_1-$unidirectional Euler alignment system \eqref{e:CSHansatz}, in the spirit of \cite{OnsagerEuler} and \cite{LS-energyNS} for the Euler and Navier-Stokes equations respectively.\\
\textbf{Remark:} Notice that $C_{x_1}^1(\T^n)\cap X_{\a}(\T^n)$ regularity in space will be enough to apply integration by parts in the weak formulation.\\

Let $E_{Q}(t)$ denote the energy associated to scales $\l_q$ for $q\le Q$ in the $x_1-$direction of the flow, and let $E(t)$ denote the total energy:
\[
E_{Q}(t):=\frac12 \int_{\T^{n-1}}E_{Q}^{x_1}(x_{-},t)\mbox{d}x_{-},\quad \quad 
E(t):=\frac12\int_{\T^{n-1}} E^{x_1}(x_{-},t)\mbox{d}x_{-},
\]
where
\[
E_{Q}^{x_1}(x_{-},t):= \int_{\T} \frac{ S_Q^{x_1}(\rho u)S_Q^{x_1}(\rho u)}{S_Q^{x_1}(\rho)}\dx_1, \qquad \text{and} \qquad E^{x_1}(x_{-},t):=\int_{\T} \rho u^2\dx_1,
\]
with $x_{-}:=(x_2,\ldots,x_n)\in\T^{n-1}$ and $\mbox{d}x_{-}:=\dx_2\ldots\dx_n$. Defining the auxiliary function  
$$U^{x_1}(x,t):=\frac{S_{Q}^{x_1}(\rho u)}{S_Q^{x_1}(\rho)}(x,t),$$
and putting $\f=S_Q^{x_1}U^{x_1}\in C(0,T;C_{x_1}^{\infty}(\T^n)\cap X_{\a}(\T^n) ) $ as test function into the weak formulation of the momentum equation \eqref{e:moment}, we obtain that
\begin{multline}\label{e:Eaux}
2 E_Q(s)\Big|_{0}^{t}-\int_{0}^{t}\int_{\T^n} S_Q^{x_1}(\rho u) \p_t U^{x_1}\dx\ds\\  = \int_0^t \int_{\T^n} S_Q^{x_1}(\rho u^2)  \p_1 U^{x_1} \dx\ds + \int_0^t \int_{\T^n} S_Q^{x_1}( \rho \cC_{\a}(\rho, u))U^{x_1}\dx\ds.
\end{multline}
On the other hand, we can rewrite the definition of $E_Q(t)$ using the weak formulation of the density equation \eqref{e:weakd} as follows
\begin{align*}
E_Q(s)\Big|_{0}^{t}&=\frac12 \int_{\T^n}S_Q^{x_1}(\rho ) (U^{x_1})^2\dx\Big|_{0}^{t}=\frac12 \int_{\T^n}\rho S_Q^{x_1}\left((U^{x_1})^2\right)\dx\Big|_{0}^{t}\\
&=\frac12\int_0^t\int_{\T^n} \rho  S_Q^{x_1}\left(\p_t(U^{x_1})^2\right)\dx\ds + \frac12\int_0^t\int_{\T^n} \rho u  S_Q^{x_1}\left(\p_1(U^{x_1})^2\right)\dx\ds\\
&=\int_0^t\int_{\T^n} S_Q^{x_1}(\rho u) \p_t U^{x_1}\dx\ds + \int_0^t\int_{\T^n} S_Q^{x_1}(\rho u) U^{x_1} \p_1 U^{x_1}\dx\ds.
\end{align*}
Subtracting the above from \eqref{e:Eaux}, we obtain the energy budget relation at scales $q\leq Q:$
\begin{equation}\label{e:Energybudget}
E_Q(t)-E_Q(0)=\int_0^{t}\Pi_Q(s)\ds-\e_Q(t).
\end{equation}
Here $\Pi_Q(s)$ is the flux through scales of order $Q$ due to the nonlinearity, defined by
\begin{equation*}
\Pi_Q(s):=\int_{\T^{n-1}} \Pi_Q^{x_1}(x_{-},s) \mbox{d}x_{-}, 
\end{equation*}
where
\begin{equation*}
\Pi_Q^{x_1}(x_{-},s):=\int_{\T} \left(S_Q^{x_1}(\rho u^2)- U^{x_1} S_Q^{x_1}(\rho u) \right)  \p_1 U^{x_1} \dx_1,
\end{equation*}
and $\e_Q(t)$ represent the change in energy due to the alignment term given by
\[
\e_Q(t):=   \int_{\T^{n-1}} \e_Q^{x_1}(x_{-},t)\mbox{d}x_{-}, \qquad \text{where} \qquad \e_Q^{x_1}(x_{-},t):= - \int_0^t \int_{\T} S_Q^{x_1}( \rho \cC_{\a}(\rho, u))U^{x_1}\dx_1\ds.
\]
Finally, we also denote
\[
\e(t):= \int_{\T^{n-1}}\e^{x_1}(x_{-},t)\mbox{d}x_{-},
\]
where
\[
\e^{x_1}(x_{-},t):=\frac12 \int_0^t\int_{\R^n\times \T}\rho(x)\rho(y)\frac{|u(x)-u(y)|^2}{|x-y|^{n+\a}}\dy\dx_1\ds.
\]
We aim to show that for appropriate $(\rho, u)$ and all $t \in (0, T)$, we have (as $Q\to \infty$) that
\begin{equation}\label{e:limits}
E_Q(t)\to E(t), \qquad \int_0^t \Pi_Q(s)\ds \to 0, \qquad \e_Q(t)\to \e(t).
\end{equation}
These convergences will immediately imply that the energy balance relation \eqref{e:ee2}  holds for $(\rho, u).$
Instead of prove the above, we will check that
\begin{equation}\label{e:unilimits}
E_Q^{x_1}(x_{-},t)\to E^{x_1}(x_{-},t), \qquad \int_0^t \Pi_Q^{x_1}(x_{-},s)\ds \to 0, \qquad \e_Q^{x_1}(x_{-},t)\to \e^{x_1}(x_{-},t),
\end{equation}
which clearly implies \eqref{e:limits} and consequently the energy balance.
Now, it was already shown in \cite{Leslie-weak} that \eqref{e:unilimits} holds (taking $x_{-}\in\T^{n-1}$ as a frozen parameter) under the assumption that
\[
\rho(\cdot,x_{-},\cdot)\in L^3(0,T;B_{3,\infty}^{1/3}(\T)), \qquad u(\cdot,x_{-},\cdot)\in L^3(0,T;B_{3,c_0}^{1/3}(\T)).
\]
So, the claimed limit then follows by the dominated convergence theorem taking
\[
\rho(t,x_1,x_{-})\in L^3(0,T;B_{3,\infty}^{1/3}(\T)\times L^3(\T^{n-1})), \qquad u(t,x_1,x_{-})\in L^3(0,T;B_{3,c_0}^{1/3}(\T)\times L^3(\T^{n-1})).
\]

To sum up, we have reduced our unidirectional setting to a well-understood problem in 1D. Recalling definition \eqref{e:aniBesov} and the fact that $B_{a,2}^{0}(\T^n)\subset L^{a}(\T^n)$ if $a\geq 1$, we have proved that the $\rho u^2-$energy relation \eqref{e:ee2} holds under the assumption that
\[
\rho\in L^3(0,T;\cB_{3,\infty}^{1/3}(\T^n)), \qquad u\in L^3(0,T;\cB_{3,c_0}^{1/3}(\T^n)).
\]
In order to obtain the $\rho-$energy relation \eqref{e:ee1}, taking $\f=S_Q^{x_1}S_Q^{x_1}\rho\in C(0,T;C_{x_1}^{\infty}(\T^n)\cap X_{\a}(\T^n) ) $ as test function into the weak formulation of the density equation \eqref{e:weakd}, we obtain (using integration by parts in the time variable) that
\[
\frac{1}{2}\int_{\T^n} (S_Q^{x_1}\rho)^2 \dx - \frac{1}{2}\int_{\T^n} (S_Q^{x_1}\rho_0)^2 \dx=\int_0^t\int_{\T^n}S_Q^{x_1}(\rho u) S_Q^{x_1}\p_1 \rho\dx\ds.
\]
Manipulating the right-hand side we get
$$\frac{1}{2}\int_{\T^n} (S_Q^{x_1}\rho)^2 \dx\Big|_{0}^t=\int_0^t\int_{\T^n}\left(S_Q^{x_1}(\rho u)-S_Q^{x_1}\rho S_Q^{x_1}u\right) S_Q^{x_1}\p_1 \rho\dx\ds -\frac12\int_0^t\int_{\T^n}S_Q^{x_1}\p_1 u  (S_Q^{x_1}\rho)^2.$$
Now, taking $\f=S_Q^{x_1}\tilde{\f}$ with $\tilde{\f}\in C^{\infty}(\T^n\times[0,T])$ on  \eqref{e:compat2}, we arrive to the compatibility condition
\[
S_Q^{x_1} e= S_Q^{x_1} \p_1 u-\L^{\a}S_Q^{x_1} \rho,
\]
which allow us to eliminate $S_Q \p_1 u$ from the last term. Then, we obtain
\begin{multline*}
\frac12\int_{\T^n} (S_Q^{x_1}\rho)^2 \dx\Big|_{0}^t=\int_0^t\int_{\T^n}\left(S_Q^{x_1}(\rho u)-S_Q^{x_1}\rho S_Q^{x_1}u\right) S_Q^{x_1}\p_1 \rho\dx\ds \\
-\frac12\int_0^t\int_{\T^n}\left(S_Q^{x_1}e +\L^{\a}S_Q^{x_1} \rho \right)(S_Q^{x_1}\rho)^2\dx\ds.
\end{multline*}
We aim to show that for appropriate $(\rho, u)$ and all $t \in (0, T)$, we have (as $Q\to \infty$) that
\begin{equation}\label{e:lim1energyrho}
\int_0^t\int_{\T^n}\left(S_Q^{x_1}(\rho u)-S_Q^{x_1}\rho S_Q^{x_1}u\right) S_Q^{x_1}\p_1 \rho\dx\ds\to 0,
\end{equation}
and
\begin{equation}\label{e:lim2energyrho}
\int_0^t\int_{\T^n} \L^{\a}S_Q^{x_1} \rho (S_Q^{x_1}\rho)^2\dx\ds\to \frac12\int_0^t \int_{\T^n\times \R^n} \left(\rho(x)+\rho(y)\right) \frac{|\rho(x)-\rho(y)|^2}{|x-y|^{n+\a}}\dy\dx\ds,
\end{equation}
because the other terms tend to their natural limits.
These convergences will immediately imply that the energy balance relation \eqref{e:ee1}  holds for $(\rho, u).$
Instead of proving the above, we will check 
\begin{equation}\label{e:lim1energyrhoaux}
\int_0^t\int_{\T}\left(S_Q^{x_1}(\rho u)-S_Q^{x_1}\rho S_Q^{x_1}u\right) S_Q^{x_1}\p_1 \rho\dx_1\ds\to 0,
\end{equation}
and
\begin{equation}\label{e:lim2energyrhoaux}
\int_0^t\int_{\T} \L^{\a}S_Q^{x_1} \rho (S_Q^{x_1}\rho)^2\dx_1\ds\to \frac12\int_0^t \int_{\T\times \R^n} \left(\rho(x)+\rho(y)\right) \frac{|\rho(x)-\rho(y)|^2}{|x-y|^{n+\a}}\dy\dx_1\ds,
\end{equation}
which clearly implies \eqref{e:lim1energyrho}, \eqref{e:lim2energyrho} and consequently the energy balance.
Now, it was already shown in \cite[Proposition 4.13]{Leslie-weak} that \eqref{e:lim1energyrhoaux} and \eqref{e:lim2energyrhoaux}  hold under the assumption that
\[
\rho(\cdot,x_{-},\cdot)\in L^a(0,T;B_{a,\infty}^{\sigma}(\T)), \qquad u(\cdot,x_{-},\cdot)\in L^b(0,T;B_{b,c_0}^{\tau}(\T)), \qquad \frac{2}{a}+\frac{1}{b}=2\sigma+\tau=1.
\]
As before, we have reduced our unidirectional setting to a well-understood problem in 1D. In particular, recalling definition \eqref{e:aniBesov}, we have proved that the $\rho-$energy relation \eqref{e:ee1} holds under the same assumption as before.

To sum up, our result provides a sufficient regularity condition on $(\rho, u)$ to guarantee  that energy equailities \eqref{e:ee1}, \eqref{e:ee2} hold under the anisotropic condition:
\begin{equation}\label{e:Onsagerclass}
\rho\in L^3(0,T;\cB_{3,\infty}^{1/3}(\T^n)), \qquad u\in L^3(0,T;\cB_{3,c_0}^{1/3}(\T^n)).
\end{equation}
Since any weak solution belongs to $L^{\infty}(0,T;L^{\infty}(\T^n))\cap L^2(0,T,H^{\a/2}(\T^n))$, by interpolation (see \cite[Proposition 2.71]{fourier}) we have that $(\rho,u)$ belongs to
\[
L^3(0,T;B_{3,2}^{\a-n/6}(\T^n)).
\]
Now \eqref{e:Onsagerclass} is automatically satisfied if $\a\in[1,2)$ and dimension $n=1,2,3$ or $4$ because the following chain of inclusions hold:
\[
B_{3,2}^{\a-n/6}(\T^n)\subset B_{3,3}^{\a-n/6}(\T^n)\subset B_{3,3}^{1/3}(\T^n)\subset  \cB_{3,3}^{1/3}(\T^n) \subset  \cB_{3,\infty}^{1/3}(\T^n) \subset  \cB_{3,c_0}^{1/3}(\T^n).
\]
Notice that for higher dimension $5\leq n\leq 9$ we have an unconditional result if $\a\in[(n+2)/6,2).$
Otherwise, the energy equalities hold under the anisotropic condition \eqref{e:Onsagerclass}.

\section{On the structure of Limiting Flocks}\label{s:Limitflock}
In this section we focus on the study of the long-time behavior of solutions in the range $\a\in[1,2).$
For the subcritical case ($\a>1$) we have unconditional global existence of classical solutions and strong flocking. As we have seen, for the critical case ($\a=1$) we have conditional global existence of classical solutions under some extra condition and unconditional global existence of weak solutions. 

In any case, the long time behavior is characterized by convergence to a flocking
state, by which we understand alignment to a constant velocity $u \to \bar{u}$, and stabilization of density to a
traveling wave $\rho(x,t)\to\rho_\infty(x-t\bar{u})$.
Although  the limiting velocity $\bar{u}$ is prescribed from the initial condition, the shape of the limiting density profile  $\rho_{\infty}$ is an emergent quantity. In general,  try  to establish the exact shape of the limiting density can be a very challenging problem. To overcome this, a first try could be to quantify how far the limiting density  deviates from a uniform
distribution.

In 1D, the authors of \cite{LS-entropy} provide a series of estimates that show how far the limiting density $\rho_{\infty}$ is from the uniform distribution $m=(2\pi)^{-1}\cM$ in the $L^1(\T)$ metric. 
Note that the same proof works line by line for  unidirectional flows in the multi-dimensional setting. Then, in terms of the $L^1(\T^n)$ metric, we have for free that:
\begin{equation*}\label{e:limitingstate1}
(e_0=0) \qquad |\rho(t)-m|_{1} \lesssim  |\rho_0-m|_{1} \, e^{-c t}, \qquad \text{for all } t\geq 0,
\end{equation*}
and
\begin{equation*}\label{e:limitingstate1}
(e_0\neq 0) \qquad |\rho_{\infty}-m|_{1} \lesssim |e_0|_{\infty}. 
\end{equation*}
The result of \cite{LS-entropy} works for local and global kernels. However, the rest of this section just apply to global kernels (smooth or singular) due to the fact that a key point in our proof will be a uniform lower bound for the density. To sum up, the main result of this sections is an improvement of the above result to any $L^p(\T^n)$ metric with $1\leq p<\infty.$ That is,
\begin{equation*}\label{e:limitingstatep}
(e_0=0) \qquad |\rho(t)-m|_{p} \lesssim  |\rho_0-m|_{p} \, e^{-c t}, \qquad \text{for all } t\geq 0,
\end{equation*}
and
\begin{equation*}\label{e:limitingstatep}
(e_0\neq 0) \qquad |\rho_{\infty}-m|_{p} \lesssim |e_0|_{\infty}. 
\end{equation*}
\textbf{Remark:} Notice that the main result of this section works even for the critical case $(\a=1)$ because all the estimates survive the limiting procedure we use to construct weak solutions.\\

The main tool in establishing results of this section is the use of energy estimates for the shifted density $r(\cdot,t):= \rho(\cdot,t) - m$. As $\rho_t =- \p_1(u \rho),$ the equaiton for the new variable is given by
\begin{equation}\label{e:eqshifted}
r_t = - \p_1( ur) - m \p_1 u .
\end{equation}
Let $p\geq 2$, multiplying \eqref{e:eqshifted} by $p |r|^{p-2}r$ we have:
\begin{align*}
\ddt |r(t)|_{p}^p&= p \int_{\T^n} |r|^{p-2}r  r_t \dx\\
&=-p\int_{\T^n} |r|^{p-2}r\left( \p_1 (ur) +m \p_1 u\right)\dx \\
&=p \int_{\T^n} \p_1 (|r|^{p-2}r) u r \dx -p m \int_{\T^n} |r|^{p-2}r \p_1 u \dx.
\end{align*}
Assuming that $p$ is an even number. That is, let $p=2q$ for any $q\in\N$, we have (applying integration by parts) that
\begin{align*}
\frac{1}{2q}\ddt |r(t)|_{2q}^{2q}&=  \int_{\T^n} \p_1 (r^{2q-1}) u r \dx - m \int_{\T^n} r^{2q-1} \p_1 u \dx\\
&=(2q-1)\int_{\T^n} r^{2q-1} u  \dx - m \int_{\T^n} r^{2q-1} \p_1 u \dx\\
&=-\frac{2q-1}{2q}\int_{\T^n}r^{2q}\p_1 u \dx- m \int_{\T^n} r^{2q-1} \p_1 u \dx\\
&=-\int_{\T^n} \p_1 u\left(\frac{2q-1}{2q}r^{2q}+m r^{2q-1} \right) \dx.
\end{align*}
Replace $\p_1 u = e + \L^{\a} \rho$, here $\L^{\a} \rho = \L^{\a} r$ is understood in the difference form. Then
\begin{equation}\label{e:general2qnorm}
\frac{1}{2q} \ddt |r(t)|_{2q}^{2q} = - \int_{\T^n} e \left(\frac{2q-1}{2q}r^{2q}+m r^{2q-1} \right) \dx -  \int_{\T^n} \left(\frac{2q-1}{2q}r^{2q}+m r^{2q-1} \right) \L^{\a} r \dx.
\end{equation}
Symmetrizing in the last term we obtain
\begin{multline*}
2\int_{\T^n} \left(\frac{2q-1}{2q}r^{2q}(x)+m r^{2q-1}(x) \right) \L^{\a} r \dx \\
=  \int_{\T^{2n}} \left[\left(\frac{2q-1}{2q}r^{2q}(x)+m r^{2q-1}(x) \right) -\left( \frac{2q-1}{2q}r^{2q}(y)+m r^{2q-1}(y) \right)\right] (r(x) - r(y)) \phi_{\a}(x-y) \dy \dx \\
=  \int_{\T^{2n}} \Psi_q(x,y) |r(x) - r(y)|^{2q} \phi_{\a}(x-y) \dy \dx,
\end{multline*}
where
\[
\Psi_q(x,y):=\frac{2q-1}{2q}\frac{r^{2q}(x)-r^{2q}(y)}{\left[r(x)-r(y)\right]^{2q-1}} + m \frac{r^{2q-1}(x)-r^{2q-1}(y)}{\left[r(x)-r(y)\right]^{2q-1}}.
\]
Notice that for $q=1,$ this expression (as $r(\cdot)=\rho(\cdot)-m$) reduces to 
\begin{equation*}
\Psi_1(x,y)=\frac{r(x)+ r(y)}{2} + m = \frac{\rho(x)+ \rho(y)}{2}\geq \rho^{-}.
\end{equation*}
In the case $q=2$, using as before that $r(\cdot)=\rho(\cdot)-m$ and taking $\rho(x), \rho(y)$ as frozen parameters, we can interpret $\Psi_2(x,y)\equiv \tilde{\Psi}_2(m)$ as a quadratic equation in terms of $m$. Then, evaluating at the minimum, we get the following uniform lower bound
\begin{equation*}
\Psi_2(x,y)\geq \frac{1}{4}\frac{\rho^-}{\rho^+}\rho^-.
\end{equation*}
A similar type of bound can be obtained for any $q\in \N$. To do that, we define the auxiliary functions $X:=r(x)$ and $Y:=r(y)$. Then, since $r(\cdot)=\rho(\cdot)-m\in [\rho^{-}-m, \rho^{+}-m]$,  the above expression $\Psi_q(x,y)$ can be written as
\[
\frac{2q-1}{2q}\frac{X^{2q}-Y^{2q}}{(X-Y)^{2q-1}} + m \frac{X^{2q-1}-Y^{2q-1}}{(X-Y)^{2q-1}},
\]
with $\rho^{-}-m\leq X,Y\leq \rho^{+}-m.$
\begin{corollary}
Let $q\in \N$, there exist a  positive constant $C(q)>0$ such that the following holds:
\begin{equation}\label{e:XYbound}
f_q(X,Y):=\frac{2q-1}{2q}\frac{X^{2q}-Y^{2q}}{(X-Y)^{2q-1}} + m \frac{X^{2q-1}-Y^{2q-1}}{(X-Y)^{2q-1}}\geq C(q) \rho^{-}.
\end{equation}
Note that $C(q)\to 0$ as long as $q\to \infty$.
\end{corollary}
\begin{proof} As we are working with a continuous function on a bounded domain, it is enough to prove that $f_q(X,Y)$ is strictly positive to obtain our result. That is, condition
\begin{equation}\label{e:positive}
f_q(X,Y)>0,
\end{equation}
give us that there exists a strictily positive constant $C_q>0$ (small enough and depending of $q$) such that $f_q(X,Y)\geq C_q>0.$ So, finally taking $C(q):=C_q/\rho^-$ we have proved our goal \eqref{e:XYbound}.

Let me start with the case $X>Y$. Under this extra condition, it is clear that $(X-Y)^{2q-1}>0$ and to obtain condition \eqref{e:positive}  is equivalent to check that the following bounds hold:
\begin{equation}\label{e:XgreaterY}
\frac{2q-1}{2q}\left(X^{2q}-Y^{2q}\right) + m \left(X^{2q-1}-Y^{2q-1}\right)>0.
\end{equation}
In addition, we distinguish between the following three cases:
\begin{itemize}
	\item Case $X>Y>0$: This case follows trivially because $X^{2q}>Y^{2q}$ and $X^{2q-1}>Y^{2q-1}.$		
	\item Case $0>X>Y$: In this case we only have that $X^{2q-1}>Y^{2q-1}$. So, inequality \eqref{e:XgreaterY} can be written in a more convenient way as
			\begin{equation}\label{e:bothnegative}
			X^{2q-1}\left(\frac{2q-1}{2q} X +m\right) >Y^{2q-1}\left(\frac{2q-1}{2q} Y +m\right).
			\end{equation}
			Now, using that $X,Y\geq\rho^{-}-m$ we get for all $q\in \N$ that
			\[
			\min\left\lbrace\frac{2q-1}{2q} X +m, \frac{2q-1}{2q} Y +m \right\rbrace\geq \frac{2q-1}{2q} \rho^{-}+\frac{m}{2q}\geq \rho^{-}.
			\]
			Consequently, combining the above with the fact that $X^{2q-1}>Y^{2q-1},$ we have proved \eqref{e:bothnegative}. 
	\item Case $X>0>Y$: Using that $X>0$, it is clear that inequality \eqref{e:XgreaterY} is obtained if the following bound holds
	\[
	-Y^{2q-1}\left( \frac{2q-1}{2q} Y +m\right)>0
	\]
	Now, we study each term separately. The first one is clearly positive because $Y<0$. For the other, we only need to use, as before, that $Y\geq \rho^{-}-m$. So, it is clear that both terms are positive and consequently we have proved our goal.
\end{itemize}
The proofs is completly similar for the case $X<Y$. We omit the details for the sake of brevity. 
\end{proof}

Then the last term of \eqref{e:general2qnorm} becomes uniformly dissipative regardless of any bounds on the density.
With this in mind, we have
\[
\frac{1}{2q} \ddt |r(t)|_{2q}^{2q} \leq  - \int_{\T^n} e \left(\frac{2q-1}{2q}r^{2q}+m r^{2q-1} \right) \dx -  c_4 \int_{\T^{2n}} |r(x) - r(y)|^{2q} \phi_{\a}(x-y) \dy \dx,
\]
with
\[
c_4:=C(q)  \rho^{-}.
\]
At this point the global communication of the model is crucial. Recall that $\phi_{\a}^{-}:=\inf_{z\in\T^n}\phi_{\a}(z)>0,$
then we get
\[
\int_{\T^{2n}} |r(y) - r(x)|^{2q} \phi_{\a}(x-y) \dy \dx \geq c_5 |r(t)|_{2q}^{2q},
\]
where
$$c_5:=2 (2\pi)^n \phi_{\a}^-.$$
The proof of the above is an inmediate consequence of the following couple of results.
\begin{lemma}
Let $f\in (L^{\infty}\cap L^{2q})(\T^n)$ with zero average, i.e. $\int_{\T^n}f(x)\dx=0$. Then, we have that 
\[
\int_{\T^{2n}} |f(x) - f(y)|^{2q}  \dy \dx \geq 2 (2\pi)^n |f|_{2q}^{2q}.
\]
\end{lemma}
\begin{proof}
Taking $X:=f(x)$, $Y:=f(y)$ and applying directly inequality \eqref{e:lowerpoly} we obtain that
\[
\int_{\T^{2n}} |f(x) - f(y)|^{2q}  \dy \dx \geq \int_{\T^{2n}} |f(x)|^{2q}+|f(y)|^{2q}  \dy \dx - \int_{\T^{2n}} P_q(f(x),f(y))\dy \dx.
\]
The result follows because the last integral vanishes due to the fact that $f$ has zero average. 
\end{proof}
\begin{corollary}
Let $X,Y\in[-l,l]$ with $0<l<\infty$ and $q\in\N$. Then, we have that
\begin{equation}\label{e:lowerpoly}
(X-Y)^{2q}\geq X^{2q}+ Y^{2q}-P_q(X,Y),
\end{equation}
where
\[
P_q(X,Y):=\begin{cases}
2XY \qquad &q=1,\\
2q\left( X^{2q-1}Y +XY^{2q-1} \right) \qquad &q\geq 2.
\end{cases}
\]
\end{corollary}
\begin{proof}
Applying directly the binomial theorem we get
\[
(X-Y)^{2q}=\sum_{k=0}^{2q}(-1)^k {2q\choose k} X^{2q-k} Y^k.
\]
In addition, we clearly have that the above can be written as
\[
(X-Y)^{2q}=X^{2q}+ Y^{2q}-P_q(X,Y) +R_q(X,Y),
\]
where
\[
R_q(X,Y):=\begin{cases}
0 \qquad & q=1,\\
6 X^2 Y^2 \qquad & q=2,\\
\sum_{k=2}^{q-2} (-1)^k{2q\choose  k } X^{2q-k} Y^{k} \qquad &q\geq 3.
\end{cases}
\]
It is clear that $R_q(X,Y)\geq 0$ for $q=1,2.$ Our next step is to prove that $R_q(X,Y)\geq 0$ for all $q\in\N$. To do that, focusing only on the non-trivial case $q\geq 3$, we rewrite the above expression in a more convenient way as follows:
\begin{equation*}\label{e:remainder}
R_q(X,Y)=(X-Y)^{2q}-X^{2q}- Y^{2q}+2q\left( X^{2q-1}Y +XY^{2q-1} \right) \qquad X,Y\in [-l,l].
\end{equation*}
Note that $R_q(X,Y)=R_q(Y,X).$ So, by simmetry, it will be enough to study the abosolute minimum of the above expression on the region $l\geq X\geq Y\geq -l.$ We start with the derivatives of $R_q(X,Y):$
\begin{align*}
\p_1 R_q(X,Y)&= \phantom{-} 2q (X-Y)^{2q-1}-2q X^{2q-1}+2q \left((2q-1)X^{2q-2}Y+Y^{2q-1} \right),\\
\p_2 R_q(X,Y)&= -2q (X-Y)^{2q-1}-2q Y^{2q-1}+2q \left(X^{2q-1}+(2q-1)X Y^{2q-2} \right).
\end{align*}
Let us start computing the critical points. That is, points such that $\p_1 R_q(X,Y)=0=\p_2 R_q(X,Y).$
Adding both expression we get the necessary condition:
\[
2q(2q-1)\left(X^{2p-2}Y +X Y^{2p-2} \right)=0.
\]
The unique possible solutions of the above in the region $X\geq Y$ are $X=Y=0$ or $X=-Y$ with $X\neq 0$. The origin $(0,0)$ is clearly a critical point but the other possibililty $X=-Y$ gives us
\begin{align*}
\p_1 R_q(X,-X)=\phantom{-}2q (2^{2q-1}-2q-1)X^{2q-1},\\
\p_2 R_q(X,-X)=-2q (2^{2q-1}-2q-1)X^{2q-1}.
\end{align*}
As $2^{2q-1}-2q-1>0$ for $q\geq 2$ we have that the only critical point and consequently a candidate to be the absolute minimum is the origin. To finish, we study the value of $R_q(X,Y)$ on the boundary.

\begin{itemize}
	\item Case $X=Y$: In this case, it is clear that
			\[
			R_q(X,X)=4q X^{2(q-1)}\geq 0.  
			\]

	\item Case $X=l$ and $-l<Y<l$: To handle the vertical line, we define the auxiliary function $V_q(Y):=R_q(l,Y),$ which derivative is given by 
			\begin{equation}\label{e:derivativeV}
			V'_q(Y)=2q\left[-(l-Y)^{2q-1}- Y^{2q-1}+ l^{2q-1} + l(2q-1) Y^{2q-2} \right].
			\end{equation}
	By the above expression, it is trivial to check that $V'_{q}(0)=0$. To conclude that $(l,0)$ is a candidate to be an aboslute minimum we show that $\sign(V'_{q}(Y))=\sign(Y)$ for $-l<Y<l.$ To do that, we have to distinguish between the following two cases:
	\begin{itemize}
		\item[$\diamond$] Case $0<Y<l:$ In this case, it is clear that 
		\begin{align*}
		l^{2q-1}> (l-Y)^{2q-1},\\
		l(2q-1) Y^{2q-2}>(2q-1)Y^{2q-1}>Y^{2q-1}.
		\end{align*}
		Consequently, we have proved that $V'_{q}(Y)>0$ if $0<Y<l.$
		
		\item[$\diamond$] Case $-l<Y<0:$ 
		Using the binomial theorem, we get that \eqref{e:derivativeV} can be written as
		\[
		V'_q(Y)=-(2q)\sum_{k=1}^{2q-3}{ 2q-1 \choose k} l^{2q-1-k} (-Y)^k.
		\]
		By the above expression, we finally get that $V'_{q}(Y)<0$ if $-l<Y<0.$
	\end{itemize}
	
	\item Case $Y=-l$ and $-l<X<l$: Similarly, to handle the horizontal line, we define the auxiliary function $H_q(X):=R_q(X,-l),$ which derivative is given by 
			\begin{equation*}\label{e:derivativeH}
			H'_q(X)=2q\left[(X+l)^{2q-1}- X^{2q-1}- l^{2q-1} - l(2q-1) X^{2q-2} \right].
			\end{equation*}
			Proceeding as before, it is not difficult to check that $\sign(H'_{q}(X))=\sign(X)$ for $-l<X<l.$ Consequently, $H_q'(0)=0$ and the point $(0,-l)$ is a candidate to be an absolute minimum.
\end{itemize}
Therefore, we have 5 possible candidates to be an absolute minimum of $R_q(X,Y)$ on $X,Y\in[-l,l].$
Evaluating we get $R_q(\pm l, 0)= R_q(0,\pm l)=R_q(0,0)=0.$ So, we finaly obtain that $R_q(X,Y)\geq 0.$
\end{proof}

Therefore, for any $q\in\N$ we have proved that
\begin{multline*}
\frac{1}{2q} \ddt |r(t)|_{2q}^{2q} \leq  - \int_{\T^n} e \left(\frac{2q-1}{2q}r^{2q}+m r^{2q-1} \right) \dx -  c_4 \int_{\T^{2n}} |r(x) - r(y)|^{2q} \phi_{\a}(x-y) \dy \dx \\
\leq  - \int_{\T^n} e \left(\frac{2q-1}{2q}r^{2q}+m r^{2q-1} \right) \dx - c_4 c_5 |r(t)|_{2q}^{2q}.
\end{multline*}

We can now perform Duhamel analysis on the size of initial $e_0$ which is conserved in time. 
Indeed, if $e_0 \equiv 0$, then the first term in the above expression drops out completely and we get
\begin{equation*}
|\rho(t)-m|_{2q} \leq |\rho_0-m|_{2q} \, e^{-c_4 c_5 t}.
\end{equation*}
For general $e_0$, we have
\begin{align*}
\frac{1}{2q} \ddt |r|_{2q}^{2q} &\leq \frac{2q-1}{2q} |e|_{\infty} |r|_{2q}^{2q}+m |e|_{\infty}|r|_{2q-1}^{2q-1} -  c_4 c_5 |r|_{2q}^{2q}\\
&\leq |e|_{\infty}\left((1+m)\frac{2q-1}{2q} |r|_{2q}^{2q}+\frac{m}{2q} (2\pi)^n \right) -  c_4 c_5|r|_{2q}^{2q},
\end{align*}
where in the last step we have used  Young's inequality. Now, applying \eqref{e:uniformentropy} we finally get 
\[
\frac{1}{2q} \ddt |r(t)|_{2q}^{2q} \leq - \left( c_4 c_5 -(1+m)\frac{2q-1}{2q}\frac{\rho^+}{\rho^-}|e_0|_{\infty}\right) |r(t)|_{2q}^{2q}+\frac{m}{2q} (2\pi)^n  \frac{\rho^+}{\rho^-}|e_0|_{\infty}.
\]
By Gronwall's lemma we obtain the bound
\begin{equation}\label{e:Gronwall}
|r(t)|_{2q}^{2q} \leq |r_0|_{2q}^{2q}e^{-2q X t } +\frac{Y}{X}(1-e^{-2qXt}),
\end{equation}
with
\begin{align*}
X:=& c_4 c_5 -(1+m)\frac{2q-1}{2q}\frac{\rho^+}{\rho^-}|e_0|_{\infty},\\
Y:=& \frac{m}{2q} (2\pi)^n  \frac{\rho^+}{\rho^-}|e_0|_{\infty}.
\end{align*}
By \eqref{e:Gronwall}, it is clear that we have control for all time over $|\rho(t)-m|_{2q}$ depending only on the inial data if $X>0$. So, a natural question is: under what condtions is $X>0?$
Note that $X>0$ if
\begin{equation*}\label{e:smallness}
|e_0|_{\infty}\leq C(q)\frac{2q}{2q-1}\frac{\rho^-}{\rho^+}\frac{\rho^{-}}{1+m}2(2\pi)^n \phi_{\a}^{-},
\end{equation*}
where $C(1)=1, C(2)=\frac{1}{4}\frac{\rho^-}{\rho^+}$ and $C(q)$ decreases to zero as long as $q$ increases their value.

In that case, taking $(\rho_0,u_0)$ such that
\begin{equation*}
|e_0|_{\infty}=\e \left[  C(q)\frac{2q}{2q-1}\frac{\rho^-}{\rho^+}\frac{\rho^{-}}{1+m}2 (2\pi)^n \phi_{\a}^{-} \right],
\end{equation*}
with $0<\e<1$ we finally obtain that
\begin{equation*}\label{e:finaleven}
|\rho_{\infty}-m|_{2q}^{2q}\leq \frac{1}{2q-1}\frac{\e}{1-\e}\frac{m}{1+m}(2\pi)^n.
\end{equation*}
Finally, in order to extend the above result to $L^{p}(\T^n)$ for all $1<p<\infty$ we simply use interpolation. 

\end{document}